\documentclass[11pt]{article}

\usepackage[margin=1in]{geometry}
\usepackage{lmodern}
\usepackage{wrapfig}                      
\usepackage[numbers,sort&compress]{natbib}
\usepackage{amsthm}

\newtheoremstyle{thmstyleone}{12pt}{12pt}{\itshape}{0pt}{\bfseries}{}{.5em}{}
\newtheoremstyle{thmstyletwo}{12pt}{12pt}{\normalfont}{0pt}{\itshape}{}{.5em}{}
\newtheoremstyle{thmstylethree}{12pt}{12pt}{\normalfont}{0pt}{\bfseries}{}{.5em}{}

\usepackage{amsmath,amsfonts,amssymb,mathtools,amsthm}
\usepackage{bm,xspace,dsfont}
\usepackage{enumitem}   
\usepackage{needspace}  
\usepackage{booktabs,algorithm,algpseudocode}
\usepackage{graphicx}
\usepackage{xcolor}
\usepackage{manyfoot}
\usepackage{textcomp}
\newcommand{\va}{{\mathbf{a}}}
\newcommand{\vb}{{\mathbf{b}}}

\newcommand{\vg}{{\mathbf{g}}}

\newcommand{\vx}{{\mathbf{x}}}
\newcommand{\vy}{{\mathbf{y}}}
\newcommand{\vz}{{\mathbf{z}}}

\newcommand{\cE}{{\mathcal{E}}}

\newcommand{\cG}{{\mathcal{G}}}

\newcommand{\cO}{{\mathcal{O}}}

\newcommand{\RR}{\mathbb{R}}

\newcommand{\one}{\mathds{1}_n}

\newcommand{\pushsum}{{\sc Push-Sum}\xspace}

\newcommand{\pushdg}{{\sc Push-DIGing}\xspace}

\newcommand{\norm}[1]{\| #1 \|}
\newcommand{\ip}[1]{\left\langle#1\right\rangle}

\theoremstyle{thmstyleone}

\theoremstyle{thmstyletwo}
\newtheorem{remark}{Remark}
\newtheorem{example}{Example}

\theoremstyle{thmstylethree}
\newtheorem{assumption}{Assumption}

\newcommand{\fulcon}{E_n}

\usepackage[colorlinks,linkcolor=blue,citecolor=blue,urlcolor=blue]{hyperref}

\theoremstyle{thmstyleone}
\newtheorem{theorem}{Theorem}
\newtheorem{proposition}{Proposition}
\newtheorem{lemma}{Lemma}

\theoremstyle{thmstylethree}
\newtheorem{definition}{Definition}

\makeatletter
\g@addto@macro\th@thmstyleone{\normalsize\thm@headfont{\normalsize\bfseries}}
\g@addto@macro\th@thmstyletwo{\normalsize\thm@headfont{\normalsize\itshape}}
\g@addto@macro\th@thmstylethree{\normalsize\thm@headfont{\normalsize\bfseries}}

\AtBeginDocument{%
  \renewenvironment{proof}[1][\proofname]{\par\removelastskip%
    \pushQED{\qed}%
    \normalsize\normalfont \topsep7.5\p@\@plus7.5\p@\relax%
    \trivlist%
    \item[\hskip\labelsep%
          \itshape%
      #1\@addpunct{}]\ignorespaces%
  }{%
    \popQED\endtrivlist\@endpefalse%
  }%
}
\makeatother

\newlength{\wrapBottomVspace}
\hypersetup{bookmarksdepth=2}

\title{Decentralized Optimization over Time-Varying Row-Stochastic Digraphs}

\author{%
  Liyuan Liang\thanks{School of Mathematical Sciences, Peking University,
    Beijing, China. \texttt{liangliyuangg@gmail.com}. These authors
    contributed equally to this work.}
  \and
  Yilong Song\thanks{Academy for Advanced Interdisciplinary Studies, Peking
    University, Beijing, China. \texttt{2301213059@pku.edu.cn}. These authors
    contributed equally to this work.}
  \and
  Kun Yuan\thanks{Center for Machine Learning Research, Peking University,
    Beijing, China. \texttt{kunyuan@pku.edu.cn}. Corresponding author.}
}
\date{}

\begin{document}

\allowdisplaybreaks
\maketitle

\begin{abstract}
Decentralized optimization over directed graphs underlies applications such as
robotic swarms, sensor networks, and distributed learning. In many such systems,
the network is a \emph{Time-Varying Broadcast Network} (TVBN), in which
out-degrees are unknown and only row-stochastic mixing matrices can be
constructed. Exact convergence of decentralized optimization over TVBNs has
remained a long-standing open problem. Row-stochastic mixing converges to a
weighted average given by the limit vector of the matrix product; since this
vector depends on unpredictable future graph realizations, bias-correction
techniques that estimate it are infeasible. We develop the first decentralized
optimization algorithm that converges exactly using only time-varying
row-stochastic matrices. Its core is PULM (Pull-with-Memory), a gossip protocol
based on a different principle: a limit vector that is not yet determined can be
controlled rather than estimated. PULM interleaves row-stochastic gossip with a
communication-free adjustment in which each of the $n$ nodes anchors the weight
of its initial vector at $1/n$, achieving exponentially fast average consensus
for every admissible graph sequence. Building on PULM, PULM-DGD finds a solution
with squared gradient norm at most $\epsilon$ for smooth nonconvex objectives
within $\mathcal{O}(\epsilon^{-1}\ln(1/\epsilon))$ communication rounds,
extending decentralized optimization to highly dynamic networks.
\end{abstract}

 


\section{Introduction}

This paper investigates decentralized optimization over a network of \( n \) nodes:
\begin{align}\label{eq(main):f definition}
    \min_{x\in\RR^d}\quad f(x):=\frac{1}{n}\sum_{i=1}^n f_i(x).
\end{align}
Each objective function~\( f_i \) is accessible only by node~\( i \) and is assumed to be smooth and potentially nonconvex. The local losses $f_i$ generally differ from each other, which poses  challenges to both the design and analysis of distributed algorithms.  

Decentralized optimization eliminates the need for a central server, thereby enhancing flexibility and enabling broad applicability in peer-to-peer communication scenarios. As a result, the design of decentralized optimization algorithms is inherently shaped by the underlying communication network among nodes, which is typically modeled as a graph or characterized by a mixing matrix. This study focuses on decentralized optimization over \emph{directed graphs}, or digraphs. Directed communication naturally models numerous real-world scenarios, including robotic swarms with asymmetric linkages~\cite{saber2003agreement,shorinwa2024distributed}, sensor networks with unidirectional message transmission~\cite{sang2010link,kar2008distributed}, and distributed deep learning systems where bandwidth asymmetry constrains communication~\cite{zhang2020network,liang2024communication}.

\subsection{Time-Varying Broadcast Network}\label{sec:TVBN}
The need for distributed optimization over directed graphs arises from complex communication constraints in real-world scenarios. Depending on the nature of these constraints, the underlying digraph may exhibit various challenging properties, including (1) not being symmetric, (2) having a time-varying topology, and (3) nodes lacking knowledge of their own out-degrees. In the most demanding communication settings, all three properties must be addressed simultaneously. We refer to a network exhibiting all these characteristics as a \emph{Time-Varying Broadcast Network} (TVBN). In many practical applications, the communication can only be accurately modeled as a TVBN: 
\begin{example}[Random Radio Broadcast]
In radio communications, transmitted information is received by any node within broadcast range, and the sender has no knowledge of which nodes have received the message. The network topology varies over time as nodes enter or exit the broadcast range.
\end{example}
\begin{example}[Byzantine Attack]
A Byzantine attack occurs when a subset of agents in the system behaves maliciously or transmits corrupted information to other nodes. Nodes receiving such malicious information may attempt to ignore or discard these unreliable signals, resulting in a TVBN.
\end{example}
\begin{example}[Packet Loss and Network Failure]
When packet loss or network failure occurs, receivers obtain incomplete or corrupted messages, creating uncertainty regarding the status of message delivery. Such scenarios can be modeled as a TVBN.
\end{example}

While decentralized optimization over time-varying digraphs has been extensively studied, all existing results, to our knowledge, require nodes to be aware of their out-degrees. Building upon the push-sum protocol, seminal works~\cite{nedic2014distributed,nedic2016stochastic,nedic2017achieving} investigate decentralized algorithms over time-varying column-stochastic networks. To ensure column-stochastic mixing matrices, each node must know its out-degree, which is not feasible in highly dynamic communication environments. Even if a node correctly determines its out-degree and scales the mixing weights accordingly for its neighbors, a network failure occurring after transmission (but before reception) may prevent some neighbors from receiving the message. In this case, the effective out-degree changes unexpectedly, and column-stochasticity can no longer be guaranteed; see Figure~\ref{fig:disconnection-plot} for an illustration. Another important line of work~\cite{saadatniaki2020decentralized,nedic2025ab,nguyen2023accelerated} studies push-pull or AB algorithms over time-varying digraphs. Since these methods alternately rely on row-stochastic and column-stochastic matrices, they also require out-degree knowledge. In fact, \textit{only algorithms relying purely on row-stochastic mixing matrices are feasible in TVBNs}, since each node only needs to know its in-degree, which is naturally immune to highly dynamic communication environments; see Figure~\ref{fig:disconnection-plot} for an illustration.
\begin{figure}[t]
    \centering
    \includegraphics[width=1\linewidth]{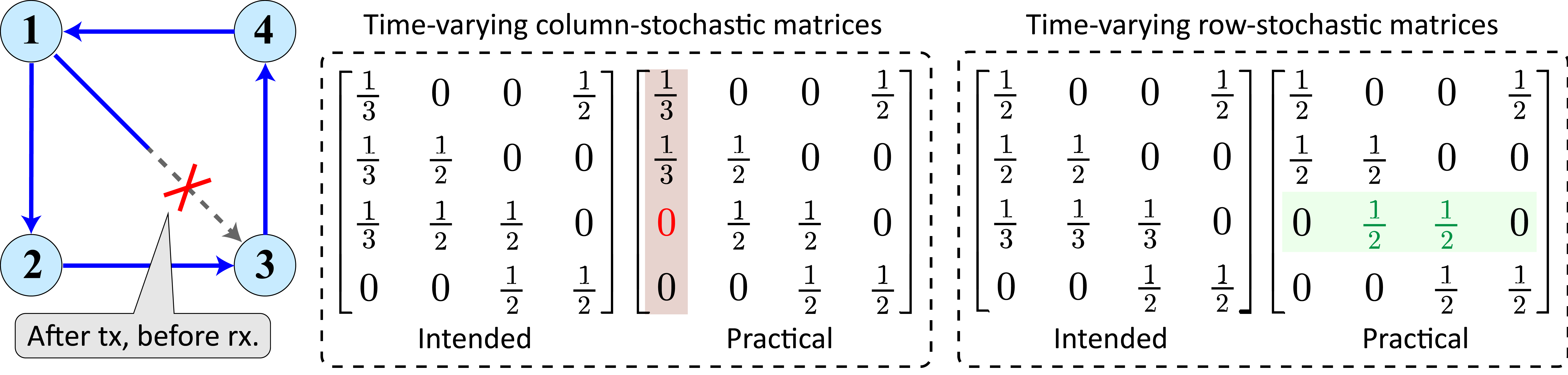}
    \caption{\small Left: A time-varying network with $n=4$ nodes experiencing unexpected network failure; ``tx'' and ``rx'' denote ``transmission'' and ``reception,'' respectively. Middle and right: The corresponding column- and row-stochastic mixing matrices. The network failure results in an incorrect column-stochastic matrix due to sudden changes in out-degree information. However, a correct row-stochastic matrix can still be constructed since it relies only on in-degree information (i.e., the actually received messages), which remains accessible.}
    \label{fig:disconnection-plot}
    \vspace{-2mm}
\end{figure}

\subsection{Open Questions and Challenges} \label{sec-open-question}
Decentralized optimization over time-varying column-stochastic  digraphs is now well understood, with foundational methods~\cite{nedic2014distributed,akbari2015distributed,nedic2016stochastic} developed over a decade ago and subsequently extended by~\cite{nedic2017achieving,saadatniaki2020decentralized}. However, developing algorithms that exactly solve problem~\eqref{eq(main):f definition} over purely time-varying row-stochastic digraphs (i.e., TVBNs) remains a long-standing open question. 

Since the out-degree is inaccessible in TVBNs, even developing decentralized algorithms to achieve average consensus is challenging. 
Assume each node $i\in[n]$ in the network initializes a vector $x_i$.
Traditional gossip algorithms~\cite{boyd2006randomized,aysal2009broadcast} can only achieve consensus among nodes in TVBNs, but not average consensus $n^{-1}\sum_{i=1}^n x_i$. The approach most related to our setting is that of~\cite{mai2016distributed}, which proposes a pre-correction strategy for \textit{static} row-stochastic mixing matrices. This strategy exploits the following property: for a nonnegative row-stochastic matrix $A$ with strong connectivity and positive diagonal entries, there exists a unique  Perron vector $\pi_A$ (i.e., $\pi_A>0$, $\pi_A^\top A=\pi_A^\top$, $\one^\top\pi_A = 1$) satisfying
$$A^k\to \one\pi_A^\top, \quad k\to \infty.$$
Using this property, \cite{mai2016distributed}  initializes $z^{(0)}_i = (n[\pi_A]_i)^{-1} x_i$ at each node and iteratively propagate $z_i$ over the network (which is equivalent to left-multiplying by $A$):
\begin{align}\label{eq-pre-correction}
\vz^{(k)} = A^k \mathrm{Diag}(\one^\top A^k )^{-1}\vx \to \one \pi_A^\top \mathrm{Diag}(n\pi_A)^{-1}\vx = n^{-1}\one\one^\top \vx, \quad k\to \infty
\end{align}
where we let $\vz^{(k)} := [(z^{(k)}_1)^\top;(z^{(k)}_2)^\top;\cdots;(z^{(k)}_n)^\top]$ and $\vx := [x_1^\top;x_2^\top;\cdots;x_n^\top]$.

However, the pre-correction strategy in \eqref{eq-pre-correction} is
intrinsically an offline, two-stage procedure. To construct the corrected
initialization $
z_i^{(0)}=(n[\pi_A]_i)^{-1}x_i,$
each agent must know the Perron vector $\pi_A$, which is a
\emph{global} property of the static mixing matrix $A$ and is generally
unavailable locally. The agents must therefore first execute an auxiliary
protocol to estimate $\pi_A$, for example, by repeatedly multiplying by $A$,
and only then restart the averaging procedure and execute the main gossip
iteration $
\vz^{(k)}=A^k\mathrm{Diag}(n\pi_A)^{-1}\vx
$
to achieve exact averaging\footnote{Strategy~\eqref{eq-pre-correction}
cannot be implemented in a decentralized, online manner. Although one could
track $V^{(k)}=AV^{(k-1)}$ with $V^{(0)}=I$, the update
$\vz^{(k)}=V^{(k)}\vb$, where
$\vb=\mathrm{Diag}(\one^\top V^{(k)})^{-1}\vx$, is still not decentralized:
each node can maintain its own row of $V^{(k)}$ by pulling the rows of its
in-neighbors, but the summation $\one^\top V^{(k)}$ aggregates entries held by all
nodes and cannot be computed by local communication.}. This offline, two-stage procedure is fundamentally incompatible with the TVBN
setting, in which $A^{(k)}$ varies over time. Even if the first stage could
produce a sequence-dependent Perron-vector estimate
\[
\pi_k=\pi\bigl(\{A^{(t)}\}_{t\le k}\bigr),
\]
where $\pi(\cdot)$ denotes an appropriate mapping, that estimate would reflect only the matrices observed in the first stage. Once the second stage begins, the agents apply the correction $(n[\pi_k]_i)^{-1}x_i$ and execute the gossip iteration, while the mixing matrices $A^{(k+1)},A^{(k+2)},\ldots$ continue to evolve. Consequently, $\pi_k$ becomes invalid immediately, because the limit vector of time-varying row-stochastic products depends on the future matrices $\{A^{(t)}\}_{t>k}$ (Section~\ref{sec:mixing-dynamics}). Since the limit vector is itself nonstationary, no fixed $\pi$ can be estimated during the first stage and remain valid throughout the second. More generally, any strategy that removes the consensus bias by estimating the limit vector is ruled out in TVBNs, and a different design principle is required.

\subsection{Related Work} In connected networks, the topology is characterized by a mixing matrix. For undirected networks, constructing symmetric and \textit{doubly stochastic} mixing matrices is straightforward. Early decentralized algorithms for such settings include decentralized gradient descent (DGD)~\cite{nedic2009distributed}, diffusion~\cite{chen2012diffusion}, and dual averaging~\cite{duchi2011dual}. These methods, however, exhibit bias under data heterogeneity~\cite{yuan2016convergence}. To address this limitation, advanced algorithms have been developed based on explicit bias-correction~\cite{shi2015extra,yuan2018exact,li2019decentralized} and gradient tracking~\cite{xu2015augmented,di2016next,nedic2017achieving,qu2017harnessing}. Extending these algorithms to time-varying undirected networks~\cite{nedic2017achieving,di2016next,maros2018panda} is natural, as time-varying doubly stochastic mixing matrices readily preserve essential average consensus properties.

For directed networks, constructing doubly stochastic matrices is generally infeasible. When the out-degree of each node is accessible, column-stochastic mixing matrices can be constructed~\cite{nedic2014distributed,tsianos2012push}. Decentralized algorithms using such column-stochastic matrices are well studied. They leverage the push-sum technique~\cite{kempe2003gossip,tsianos2012push} to correct the bias in column-stochastic communications and achieve global averaging of variables or gradients. While the subgradient-push algorithm~\cite{nedic2014distributed,tsianos2012push} guarantees convergence to optimality, its sublinear rate persists even under strong convexity. Subsequent work including EXTRA-push~\cite{zeng2017extrapush}, D-EXTRA~\cite{xi2017dextra}, and ADD-OPT~\cite{xi2017add} attains faster convergence by explicitly mitigating heterogeneity, but relies on the fixed Perron vector of a \textit{static} column-stochastic matrix and hence does not directly extend to time-varying digraphs. In contrast, algorithms that estimate the Perron vector online through the push-sum ratio remain applicable to time-varying column-stochastic digraphs. Representative examples are Push-DIGing~\cite{nedic2017achieving}, which combines push-sum with gradient tracking and retains linear convergence, and SONATA~\cite{scutari2019distributed}, which uses perturbed push-sum to handle nonconvex objectives and admits general column-stochastic matrices. Recent work~\cite{liang2023towards} has established effective topology metrics, lower bounds, and optimal algorithms for decentralized optimization over column-stochastic digraphs.

When only in-degree information is available, one can construct row-stochastic mixing matrices. Diffusion~\cite{chen2012diffusion,sayed2014adaptive} was among the earliest algorithms using row-stochastic matrices, but converges only to a Pareto-optimal solution rather than the global optimum. Just as push-sum underpins column-stochastic algorithms, the Pull-Diag gossip protocol~\cite{mai2016distributed} serves as an effective technique to correct the bias caused by row-stochastic communications. Reference~\cite{mai2019distributed} first adapted distributed gradient descent to this setting. Subsequently, gradient tracking techniques were extended to the row-stochastic scenario by~\cite{xi2018linear,li2019row,FROST-Xinran,liangachieving}, while momentum-based variants were developed in~\cite{ghaderyan2023fast,lu2020nesterov}. All of these algorithms rely on an estimate of the Perron vector of a \textit{static} row-stochastic matrix, which is not well defined when the matrix changes at every round. To the best of our knowledge, no existing algorithm achieves exact convergence to the solution of problem~\eqref{eq(main):f definition} using purely row-stochastic mixing matrices over time-varying networks, due to the challenges discussed in Section~\ref{sec-open-question}.

In digraphs where both in-degree and out-degree information are available, the Push-Pull/AB methods~\cite{pu2020push,xin2018linear,you2025stochastic,liang2025linear} can solve problem~\eqref{eq(main):f definition} by alternately using column-stochastic and row-stochastic mixing matrices. Since neither mixing matrix needs to be normalized by an estimated Perron vector, these methods extend to time-varying digraphs~\cite{saadatniaki2020decentralized,nedic2025ab,akgun2024projected} and typically achieve faster convergence than methods relying solely on column- or row-stochastic matrices. However, they still require knowledge of the out-degree at each node, which is unavailable in TVBNs.

\subsection{Main Results} 
\label{sec:main-results}
This paper develops the first algorithm that achieves stationarity for problem~\eqref{eq(main):f definition} using only time-varying row-stochastic mixing matrices, thereby making decentralized optimization feasible over TVBNs. The key idea is to control the limit vector rather than estimate it. The future dependence identified in Section~\ref{sec-open-question} is both the obstacle and the opportunity. \textit{A limit vector that depends on matrices not yet applied cannot be estimated; for precisely the same reason, it is not yet determined and can therefore be controlled during mixing}. Existing methods estimate a network-induced limit vector and compensate for it, either before gossip, as in the pre-correction~\eqref{eq-pre-correction}, or after gossip, as in push-sum (Remark~\ref{rmk:col-vs-row}). We propose a new gossip protocol that instead interleaves the correction with gossip: after each gossip step, every node $i$ resets the weight of its own initial vector $x_i$ in its state to $1/n$. This anchoring is local, communication-free, and independent of out-degrees, and subsequent gossip diffuses the anchored weights until the mixing process reaches the uniform average for every admissible graph sequence. The proposed protocol thus departs from Perron-correction methods in principle, while matching their $\mathcal{O}(d+n)$ per-round communication and per-node storage on static networks (Table~\ref{tab:overhead}). 
Our contributions can be summarized as follows:
 
\begin{itemize}[leftmargin=3em]
    \item[C1.] \textbf{Effective average consensus protocol}. We propose \underline{\textbf{PUL}}L-with-\underline{\textbf{M}}emory \textbf{(PULM)}, a decentralized method to achieve average consensus over TVBNs with time-varying row-stochastic mixing matrices. By alternating between a standard row-stochastic gossip step and a local adjustment step, we prove that PULM converges to average consensus exponentially fast.

    \vspace{1mm}
    \item[C2.] \textbf{The first exactly converging algorithm}. Built upon PULM, we develop a decentralized gradient descent approach over TVBNs, termed \textbf{PULM-DGD}. For nonconvex and smooth optimization problems, we establish that PULM-DGD finds an $\epsilon$-stationary solution within $\cO(\epsilon^{-1}\ln(1/\epsilon))$ communication rounds. To our best knowledge, PULM-DGD is the first method that achieves stationarity using only time-varying row-stochastic mixing matrices. 
\end{itemize}

\vspace{1mm}
\textbf{Organization.} The remainder of this paper is organized as follows. Notation and assumptions are provided in Section~\ref{sec:ass}. In Section~\ref{sec:prelim}, we examine the mixing dynamics in TVBNs and derive our PULM approach for achieving distributed average consensus. Performing decentralized optimization through PULM is discussed in Section~\ref{sec:PWM-GD}, where we provide the main convergence theorems. We conduct numerical experiments to verify PULM and PULM-DGD in Section~\ref{sec:exp} and conclude in Section~\ref{sec:conclusion}.

\section{Notations and Assumptions} \label{sec:ass}

\vspace{1mm}
\noindent\textbf{Notations.}  Throughout, the number of nodes is $n\ge2$. Let \(\one\) denote the vector of all-ones of \(n\) dimensions and \(I_n\in\mathbb{R}^{n\times n}\) the identity matrix. In the context of proofs, the notation \(\fulcon:=\frac{1}{n}\one\one^\top\) is introduced for simplicity.  
 We use \([n]\) to denote the set \(\{1, 2, \ldots, n\}\). For a given vector \(v\) or  a given matrix $X$, $\mathrm{Diag}(v),\mathrm{Diag(X)}$ signify the diagonal matrix whose diagonal elements are comprised of all the entries of \(v\), or the diagonal elements of given \(X\), respectively. We define \(n \times d\) matrices
\begin{align*}
\vx^{(k)} &\hspace{-0.5mm}:=\hspace{-0.5mm} [( {x}_1^{(k)})^\top\hspace{-0.5mm}; ( {x}_2^{(k)})^\top\hspace{-0.5mm}; \cdots; ( {x}_n^{(k)})^\top\hspace{-0.5mm}]  \\
\vg^{(k)}=\nabla F(\vx^{(k)}) &\hspace{-0.5mm}:=\hspace{-0.5mm} [\nabla f_1({  x}^{(k)}_1)^\top;
\cdots;\nabla f_n({  x}^{(k)}_n)^\top] 
\end{align*}
by stacking all local variables vertically. The upright bold symbols (e.g. $\vx,\vg$) always denote network-level quantities. For vectors or matrices, we use the symbols $\le$ and $\ge$ for element-wise comparison. We use $\norm{\cdot}$ for $\ell_2$ vector norm and $\norm{\cdot}_F$ for matrix Frobenius norm. We use $\norm{\cdot}_2$ for induced $\ell_2$ matrix norm, which means $\norm{A}_2:=\max_{\norm{v}=1} \norm{Av}$. We use $\max|A|$ to represent the maximum absolute value of all the elements of matrix \(A\). Unless otherwise specified, product signs for matrices always indicate consecutive left multiplication in order, i.e., \(\prod_{i=1}^K A^{(i)}:=A^{(K)}A^{(K-1)}\cdots A^{(2)}A^{(1)}\). When a directed graph $\cG$ is classified as strongly connected, it means there exists a directed path from $i$ to $j$ for any nodes $i,j$ in $\cG$. 

\vspace{1mm}
\noindent \textbf{Gossip communication.} When node $i$ collects information $x_j\in \mathbb{R}^d, j\in N_i^{\rm{in}}$ from its in-neighbors, it can mix these vectors using scalars $a_{ij}$, producing $x_i^{\text{new}} = \sum_{j\in N_i^{\rm{in}}} a_{ij}x_j$. This process is called gossip communication. Using the stacked notation $\vx = [x_1^\top; x_2^\top; \ldots; x_n^\top]$ and $A=[a_{ij}]_{n\times n}$, the update can be written as $\vx^{\text{new}} \leftarrow A\vx$, where $A$ is called the mixing matrix. When out-degrees are accessible, $A$ can be constructed as either column-stochastic or row-stochastic. When out-degrees are inaccessible, $A$ can only be row-stochastic.

\vspace{1mm}
\noindent \textbf{Assumptions.} The following assumptions are used throughout this paper.

\begin{assumption}[\scshape \small $\tilde{B}$-Strongly Connected Graph Sequence]\label{ass:tv graph}
    The time-varying directed graph sequence $\{\mathcal{G}^{(k)}=(\mathcal{V},\ \mathcal{E}^{(k)})\}_{k\ge 0}$ satisfies the following condition: there exists an integer $\tilde{B}>0$ such that, for every $k\ge 0$, the $\tilde{B}$-step accumulated graph
    $
\mathcal{G}_{k}^{k+\tilde{B}-1}:=\Big(\mathcal{V},\bigcup_{l=k}^{k+\tilde{B}-1} \mathcal{E}^{(l)}\Big)
    $
    is strongly connected. In addition, every node has a self-loop in each graph $\mathcal{G}^{(k)}$.
\end{assumption}

\begin{assumption}[\scshape\small Rapidly Changing Broadcast Network]
\label{ass:broadcast}
For each $k\ge 0$ and each $i\in [n]$, node $i$ does not know its out-degree $d_i^{{\rm{out}}, (k)}$ in the graph $\cG^{(k)}$. Moreover, for each $k\ge 0$, the mixing matrix induced by $\cG^{(k)}$ may be used only once. Each node knows the network size $n$ and its own index $i$, and indices are distinct across nodes.
\end{assumption}

The single-use restriction on the mixing matrices in Assumption~\ref{ass:broadcast} reflects the rapidly changing topology: by the time the communication associated with $A^{(k)}$ is completed, the network has already transitioned to $\cG^{(k+1)}$, so that $A^{(k)}$ is no longer consistent with the current network structure. Such behavior is typical of highly dynamic environments such as drone swarms,
vehicular ad hoc networks, and low-earth-orbit satellite constellations, where
the link set may change within the duration of a single communication round. In contrast to out-degrees, the network size and the node indices are time-invariant and can be fixed once at deployment. They are also required by Pull-Diag-type methods~\cite{mai2016distributed,xi2018linear,FROST-Xinran,liangachieving}, which maintain an $n$-dimensional estimate of the Perron vector indexed in the same way.

\begin{definition}[\scshape\small Compatible Mixing Matrices]\label{def:compatible}
A mixing matrix $A = [a_{ij}]_{n\times n}$ is compatible with graph $\cG$ if $a_{ij}> 0 \text{ when }  (j\to i)\in \cE, \text{ and } a_{ij}=0 \text{ otherwise}.$
Under Assumption~\ref{ass:broadcast}, we require compatible matrices to be row-stochastic, i.e., $A\one = \one$, which needs only in-degree information. 
\end{definition}

\begin{assumption}[\scshape\small Lower Bounded Entries]\label{ass:matrix}
Suppose that for each $k\ge 0$, we have a mixing matrix $A^{(k)}$ that is compatible with $\cG^{(k)}$. There exists a scalar $\tau>0$ such that all nonzero entries of $A^{(k)}$ satisfy $a^{(k)}_{ij} \ge \tau$ for all $k\ge 0$.
\end{assumption}
\noindent Assumption~\ref{ass:matrix} is naturally satisfied by setting $A^{(k)}_{ij} = 1/d_i^{\mathrm{in},(k)}$ and $\tau = 1/n$, where $d_i^{\mathrm{in},(k)}$ denotes the in-degree of node $i$ in $\mathcal{G}^{(k)}$, counting the self-loop.
\begin{proposition}\label{prop:positive}
For the sequence of time-varying directed graphs coupled with compatible mixing matrices $\{\cG^{(k)}, A^{(k)}\}_{k\ge 0}$ satisfying Assumptions~\ref{ass:tv graph}--\ref{ass:matrix}, there exist an integer $0<B\le n\tilde{B}$ and a scalar $\eta\in [\tau^B,1)$ such that $\prod_{l=k}^{k+B-1} A^{(l)} \ge \eta, \forall k\ge 0$.
\end{proposition}
Proposition~\ref{prop:positive} establishes that, for any starting time $k$, the $B$-step product $\prod_{l=k}^{k+B-1}A^{(l)}$ is entrywise lower bounded by a constant $\eta>0$, with $B\le n\tilde{B}$ and $\eta\ge\tau^{B}$. This property is central to the convergence analysis: it guarantees uniform mixing over every $B$-step window, so that each node's information reaches every other node with weight at least $\eta$, uniformly in $k$. {Results of this type go back to Tsitsiklis et al.~\cite{tsitsiklis1986distributed}, and the closest published form is Lemma~2 of~\cite{nedic2009distributed}. That lemma additionally requires every persistent edge to reappear within each window of $B$ rounds, which allows a path between two nodes to be fixed in advance; Assumption~\ref{ass:tv graph} imposes only strong connectivity of the union graph over each window, so our proof instead constructs the path online, one window at a time.} The proof is given in Appendix~\ref{sec:proof-prop-positive}. We next introduce our final assumption.
\begin{assumption}[\sc Smoothness]\label{ass:smooth}
All the local functions $f_i(\cdot)$ are lower bounded, and there exists a constant $L>0$ such that for all $i\in [n]$ and all ${x}, {y}\in \RR^d$,
\begin{align*}
    \left\|\nabla f_{i}({x})-\nabla f_{i}({y})\right\| \leq L\|{x}-{y}\|.
\end{align*}
\end{assumption}

\section{Achieving Average Consensus over TVBNs}\label{sec:prelim}
In this section, we analyze how information is mixed and propagated across TVBNs using time-varying row-stochastic mixing matrices only. Let each node $i$ initialize a vector $z_i^{(0)}=x_i$. The $k$-th communication round is governed by the row-stochastic mixing matrix $A^{(k)}$. At the beginning of the $k$-th round, each node $i$ holds a vector $z_i^{(k)}$ that is a linear combination of the initial vectors:
\begin{align}
\label{eq-mixing-dynamic}
z_i^{(k)} = \sum_{j=1}^n w_{ij}^{(k)}x_j,
\end{align}
where the weights $w^{(k)}_{ij}$ are specified below through a recursion that does not depend on the vectors $\{x_j\}$. We term \eqref{eq-mixing-dynamic} the \textit{mixing dynamics}.

\subsection{Mixing Dynamics over TVBNs}
\label{sec:mixing-dynamics}
We collect the weights in~\eqref{eq-mixing-dynamic} into the \emph{distribution matrix}
\(
W^{(k)} \coloneqq [w_{ij}^{(k)}]_{i,j=1}^n \in \mathbb{R}^{n\times n},
\)
which records how much of each initial vector $x_j$ is held by each node at round $k$. Letting
$\vz^{(k)} \coloneqq \big[z_1^{(k)},\ldots,z_n^{(k)}\big]^\top\in \mathbb{R}^{n\times d}$ and $\vx \coloneqq \big[x_1,\ldots,x_n\big]^\top\in \mathbb{R}^{n\times d}$,
the mixing dynamics~\eqref{eq-mixing-dynamic} read $\vz^{(k)} = W^{(k)} \vx$.
Since $z_i^{(0)} = x_i$, we set $W^{(0)} = I_n$. Under the gossip update
$z_i^{(k+1)} = \sum_{j=1}^n a_{ij}^{(k)} z_j^{(k)}$, we define
\begin{equation}
\label{eq:A-gossip}
W^{(k+1)} = A^{(k)} W^{(k)},
\end{equation}
where $A^{(k)}=[a_{ij}^{(k)}]_{n\times n}$ is row-stochastic; by induction, $\vz^{(k)} = W^{(k)}\vx$ holds for every $\vx$. Average consensus is achieved if $W^{(k)} \to \tfrac{1}{n}\one\one^\top$ as $k\to\infty$. However, the gossip update \eqref{eq:A-gossip} cannot drive $W^{(k)}$ to $n^{-1}\one\one^\top$. Instead, it converges to a biased average:
\begin{proposition}[\sc Limiting Property]\label{prop:row-sto-linear-convergence}
    For the sequence of time-varying directed graphs coupled with compatible mixing matrices $\{\cG^{(k)}, A^{(k)}\}_{k\ge 0}$ satisfying Assumptions~\ref{ass:tv graph}--\ref{ass:matrix}, the infinite product \(\prod_{k=0}^\infty A^{(k)}\) converges, and the limit takes the form of \(\prod_{k=0}^\infty A^{(k)}=\one \pi^\top\), where \(\pi=\pi\big(\{A^{(k)}\}_{k=0}^\infty\big)\in\mathbb{R}^n\), which means that \(\pi\) is determined by the whole sequence \(\{A^{(k)}\}_{k=0}^\infty\). Moreover, the convergence occurs at an exponential rate:
    \begin{equation}
        \Bigg\|\prod_{i=0}^{k-1} A^{(i)}-\one \pi^\top\Bigg\|_F \le n^{3/2}\frac{1+\eta}{\eta(1-\eta)}\cdot\big(\sqrt[B]{1-\eta}\big)^k,
        \label{eq:A-linear-convergence}
    \end{equation}
    where $B$ and $\eta$ are defined in Proposition~\ref{prop:positive} (Proof is in Appendix~\ref{sec-app-limit-property}).
\end{proposition}
Proposition~\ref{prop:row-sto-linear-convergence} is a classical ergodicity result for products of row-stochastic matrices~\cite{tsitsiklis1984problems}; our proof follows Lemma~3 of~\cite{nedic2009distributed}. While Proposition~\ref{prop:row-sto-linear-convergence} guarantees asymptotic consensus, the limit is a $\pi$-weighted average rather than the uniform average consensus $n^{-1}\sum_{i=1}^n x_i$. Moreover, the limit vector $\pi$ cannot be exactly determined by the finite initial sequence $\{A^{(l)}\}_{l=0}^{k-1},\forall k<\infty$. In fact, different future sequences $\{A^{(l)}\}_{l=k}^\infty$ that share the same past yield different limit vectors $\pi_k=\pi\big(\{A^{(l)}\}_{l=k}^\infty\big)$. The relation between the original limit and the limit induced by the tail sequence starting from the $k$-th matrix is given by
\[
\one \pi^\top =\prod_{l=0}^\infty A^{(l)}=\prod_{l=k}^\infty A^{(l)}\prod_{l=0}^{k-1}A^{(l)}=\one \pi_k^\top \prod_{l=0}^{k-1}A^{(l)},
\]
namely $\pi^\top = \pi^\top_k \prod_{l=0}^{k-1} A^{(l)}$. Given that $\pi_k$ depends on the future matrices, the intended limit vector $\pi$ is unpredictable if we only utilize the past observations of $\{A^{(l)}\}_{l=0}^{k-1}$. This poses a fundamental challenge for the two-stage offline pre-correction strategy \eqref{eq-pre-correction}. The following remark further clarifies why average consensus can be achieved easily through time-varying column-stochastic mixing matrices, yet remains challenging when using time-varying row-stochastic mixing matrices.

\begin{remark}[\sc Column-stochastic vs\ Row-stochastic matrices]\label{rmk:col-vs-row}
Average consensus is more easily achieved using column-stochastic mixing matrices. For a nonnegative column-stochastic matrix $C$ with strong connectivity and positive diagonal entries, there exists a unique Perron vector $\pi_C$ (i.e., $\pi_C>0$, $C \pi_C=\pi_C$, $\one^\top\pi_C = 1$) satisfying $C^k\to \pi_C \one^\top \text{as } k\to \infty.$ The push-sum protocol~\cite{kempe2003gossip,tsianos2012push} achieves average consensus as follows
\begin{align}\label{eq-pre-correction-column}
\vz^{(k)} \hspace{-0.8mm}=\hspace{-0.8mm} \mathrm{Diag}(C^k \one)^{-1} C^k \vx \hspace{-0.8mm}\to\hspace{-0.8mm} \mathrm{Diag}(n\pi_C)^{-1} \pi_C \one^\top \vx\hspace{-0.8mm}=\hspace{-0.8mm} n^{-1}\one\one^\top \vx, \quad k\to \infty.
\end{align}
In contrast to row-stochastic strategy~\eqref{eq-pre-correction}, which applies \emph{pre-correction} $[\mathrm{Diag}(n\pi)]^{-1}$ prior to gossip and requires knowledge of the limit vector $\pi$ in advance, the column-stochastic approach~\eqref{eq-pre-correction-column} employs \emph{post-correction}: nodes first propagate information via gossip, then apply local corrections. Crucially, this approach does not require prior knowledge of the limit vector $\pi$, thereby enabling an online single-stage implementation that naturally extends to time-varying scenarios. Specifically, by updating $\vz^{(k)} = C^{(k)} \vz^{(k-1)}$ and $v^{(k)} = C^{(k)} v^{(k-1)}$ with $v^{(0)} = \one$, each node obtains $y_i^{(k)} = z_i^{(k)} / v_i^{(k)} \to n^{-1} \one^\top \vx$ as $k \to \infty$. This distinction between pre-correction and post-correction explains why decentralized optimization over time-varying column-stochastic digraphs is well understood, while time-varying row-stochastic digraphs remains an open problem.
\end{remark}

\begin{figure}[t]
    \centering
    \includegraphics[width=1\linewidth]{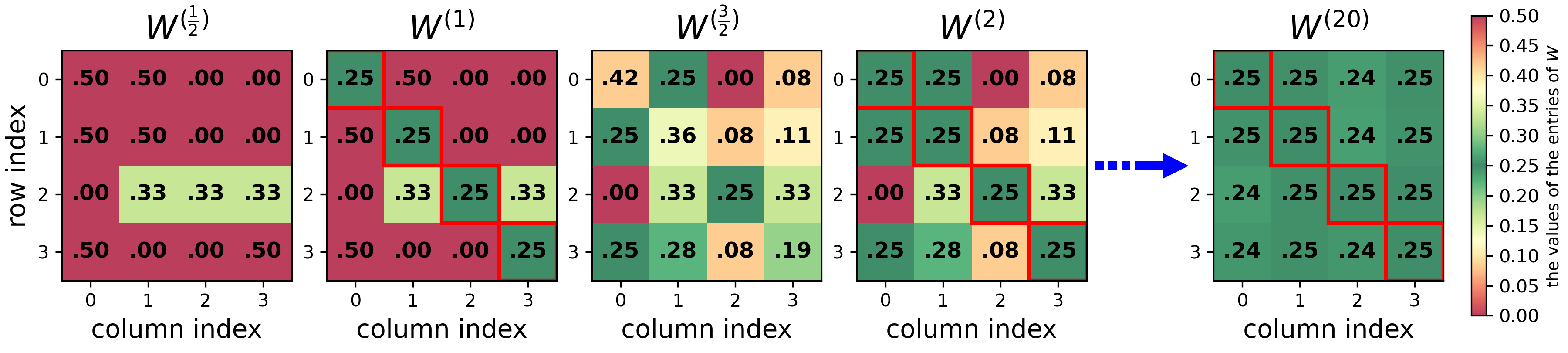}
    \caption{\small \textbf{Diffusion of the $1/n$ diagonal anchor under PULM.} Each adjustment step resets the diagonal entry to an anchor value $1/n$; subsequent gossip diffuses this anchor through the corresponding column.}
    \label{fig:adjust-gossip-vis}
\end{figure}

\subsection{The PULM Method: Algorithm Design}
\label{sec:PULM-matrix}

Our method rests on the observation made in Section~\ref{sec:mixing-dynamics}: the limit vector $\pi$ depends on the entire sequence of mixing matrices, including those not yet applied, and is therefore not determined by the past. Consequently, modifications of the intermediate $W^{(k)}$ change where the mixing process converges, and can steer its limit toward a prescribed target, here $n^{-1}\one\one^\top$. Starting from $W^{(0)} = I_n$, we interleave a gossip step with an adjust step that resets every diagonal entry to $1/n$ while leaving the off-diagonal entries untouched:
\begin{subequations}\label{eq:gossip-adjust}
\renewcommand{\theequation}{\theparentequation.\arabic{equation}}
\begin{align}
W^{(k+\frac{1}{2})} &= A^{(k)} W^{(k)}, \label{eq:gossip}\\
D^{(k)} &= \mathrm{Diag}\big(W^{(k+\frac{1}{2})}\big) - n^{-1} I_n, \label{eq:diag}\\
W^{(k+1)} &= W^{(k+\frac{1}{2})} - D^{(k)}. \label{eq:adjust}
\end{align}
\end{subequations}


\Needspace{0.38\textheight}
\begin{wrapfigure}[17]{r}{0.40\textwidth}
    \centering
    \includegraphics[width=0.36\textwidth]{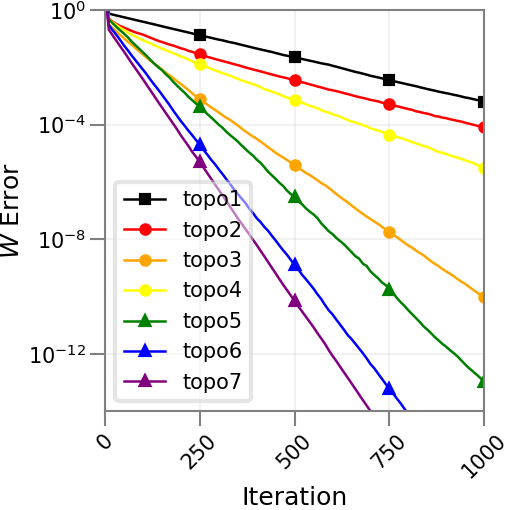}
    \caption{\small
    Convergence of $W^{(k)}$. ``$W$ Error'' denotes
    $\|W^{(k)}-n^{-1}\one\one^\top\|_F$. Details are in
    Section~\ref{sec:topo-exp}.
    }
    \label{fig:right-wrap}
\end{wrapfigure}

\noindent The adjustment is communication-free and uses only locally available
information, since $D^{(k)}$ is diagonal and node $i$ modifies the single entry
$W^{(k+\frac{1}{2})}_{ii}$. The mechanism by which recursion~\eqref{eq:gossip-adjust} drives $W^{(k)}$ toward
$n^{-1}\one\one^\top$ is straightforward. The gossip step forms convex
combinations of the rows of $W^{(k)}$ (see \eqref{eq:A-gossip}), so the entries
within each column are repeatedly mixed and progressively approach a common
value (see Proposition~\ref{prop:row-sto-linear-convergence}). Since column $j$
always retains the anchored value $W^{(k+1)}_{jj} = 1/n$, this mixing pulls the
remaining entries of that column toward $1/n$, thereby diffusing the anchor
throughout the column (see Figure~\ref{fig:adjust-gossip-vis}). Over successive
iterations, all entries converge to $1/n$. Figure~\ref{fig:right-wrap} reports a
numerical simulation confirming that the procedure converges exponentially fast
to average consensus over various time-varying networks with row-stochastic
mixing matrices. Theorem~\ref{thm:exp convergence} establishes the convergence of the PULM method. 

\begin{remark}[\sc Difference from Perron-Correction Methods]
Pull-Diag~\cite{mai2016distributed,xi2018linear,liangachieving} and PULM both let node $i$ track the $i$-th row of a matrix and use its diagonal entry, but in opposite ways. Pull-Diag tracks $A^k$ for a static $A$ and divides by $[A^k]_{ii}\to[\pi_A]_i$, which estimates the Perron vector and requires the same matrix to be applied repeatedly. PULM tracks $W^{(k)}$ and resets $W^{(k)}_{ii}$ to $1/n$, which prescribes the limit instead of estimating it (Section~\ref{sec:main-results}). As a by-product, PULM involves no division and thus avoids the instability caused by inverting small diagonal entries of $A^k$~\cite{liangachieving}.
\end{remark}


\begin{theorem}\label{thm:exp convergence}
Let $\{\cG^{(k)}, A^{(k)}\}_{k\ge 0}$ be a sequence of time-varying directed
graphs with compatible mixing matrices satisfying
Assumptions~\ref{ass:tv graph}--\ref{ass:matrix}, and let $\{W^{(k)}\}_{k\ge 0}$
be generated by~\eqref{eq:gossip-adjust} with $W^{(0)} = I_n$. Then:
\begin{enumerate}[label=(\roman*)]
    \item the sequence $\big\{\max_{i,j}\big|[W^{(k)}]_{ij}-n^{-1}\big|\big\}_{k\ge 0}$
    is non-increasing;
    \item for every $K\ge 0$, it holds that
    \begin{equation}\label{eq-pulm-convergence-rate}
    \|W^{(K)}-n^{-1}\one\one^\top\|_F\le \frac{n}{1-\eta}(1-\eta)^{K/B}.
    \end{equation}
\end{enumerate}
where the constants $B>0$ and $\eta\in(0,1)$ are defined in
Proposition~\ref{prop:positive}.
\end{theorem}

To establish convergence, we first need the following lemma.

\begin{lemma}[\sc Products of Principal Submatrices]\label{lem:submatrix-product}
For a matrix $X\in\RR^{n\times n}$ and an index $\ell\in[n]$, let
$X_{-\ell,-\ell}$ denote the $(n-1)\times(n-1)$ principal submatrix obtained
by deleting row $\ell$ and column $\ell$, let $X_{-\ell,\ell}$ denote column
$\ell$ with its $\ell$-th entry removed, and let $X_{\ell,-\ell}$ denote row
$\ell$ with its $\ell$-th entry removed. For nonnegative matrices
$A^{(0)},\ldots,A^{(s-1)}\in\RR^{n\times n}$ and any $\ell\in[n]$,
\begin{equation}\label{eq:pulm-submatrix-product}
0\ \le\ \prod_{r=0}^{s-1}A^{(r)}_{-\ell,-\ell}
\ \le\ \Big(\prod_{r=0}^{s-1}A^{(r)}\Big)_{-\ell,-\ell}
\qquad\text{entrywise.}
\end{equation}
\end{lemma}
\begin{proof}
The left inequality holds because products of nonnegative matrices are
nonnegative. For the right inequality with $s=2$, block multiplication gives
\[
\big(A^{(1)}A^{(0)}\big)_{-\ell,-\ell}
=A^{(1)}_{-\ell,-\ell}A^{(0)}_{-\ell,-\ell}
+A^{(1)}_{-\ell,\ell}A^{(0)}_{\ell,-\ell},
\]
and the second term is a nonnegative rank-one matrix. For general $s$, write
$Q:=\prod_{r=0}^{s-2}A^{(r)}$; the case $s=2$ and the induction hypothesis give
\[
\big(A^{(s-1)}Q\big)_{-\ell,-\ell}
\ \ge\ A^{(s-1)}_{-\ell,-\ell}\,Q_{-\ell,-\ell}
\ \ge\ A^{(s-1)}_{-\ell,-\ell}\prod_{r=0}^{s-2}A^{(r)}_{-\ell,-\ell},
\]
where the last step multiplies an entrywise inequality between nonnegative
matrices by the nonnegative matrix $A^{(s-1)}_{-\ell,-\ell}$.
\end{proof}
{
Intuitively, deleting coordinate $\ell$ before each multiplication removes the nonnegative contributions routed through it. Hence the product of the principal submatrices is bounded entrywise by the corresponding principal submatrix of the full product.}

\begin{proof}[Proof of Theorem~\ref{thm:exp convergence}]
Let $\delta_k:=\max_{i,j}|W^{(k)}_{ij}-n^{-1}|$ denote the largest entrywise
error at round $k$, so that part~(i) asserts $\delta_{k+1}\le\delta_k$. Fix a
column $\ell\in[n]$ and track its error with the anchored entry removed,
\[
e_\ell^{(k)}:=W^{(k)}_{-\ell,\ell}-n^{-1}\mathds{1}_{n-1}\in\RR^{n-1},
\]
where $\mathds{1}_{n-1}$ is the all-ones vector of $\RR^{n-1}$ and the
notation $W^{(k)}_{-\ell,\ell}$ is that of Lemma~\ref{lem:submatrix-product}.
For a vector, $\|\cdot\|_\infty$ denotes the largest absolute entry. The
proof proceeds in three steps: the anchor reduces each column to a
substochastic recursion; Proposition~\ref{prop:positive} makes that recursion
contract over every $B$-step block; the contraction is converted into the
Frobenius-norm bound~\eqref{eq-pulm-convergence-rate}.

\emph{Step 1: Anchoring turns each column into a substochastic recursion.}
The adjust step~\eqref{eq:adjust} changes only diagonal entries, so for
$i\neq\ell$ the entry $W^{(k+1)}_{i\ell}$ is produced by the gossip
step~\eqref{eq:gossip} alone. Using the row sums $\sum_{j}A^{(k)}_{ij}=1$, we
obtain
\begin{equation}\label{eq:pulm-column-convex}
W^{(k+1)}_{i\ell}-\frac1n
=\sum_{j=1}^n A^{(k)}_{ij}\Big(W^{(k)}_{j\ell}-\frac1n\Big),
\qquad i\neq\ell,\ k\ge0 .
\end{equation}
Every off-anchor error at round $k+1$ is therefore a convex combination of
the column-$\ell$ errors at round $k$, which yields
$|W^{(k+1)}_{i\ell}-n^{-1}|\le\max_j|W^{(k)}_{j\ell}-n^{-1}|\le\delta_k$.
The anchored entry has zero error, $W^{(k+1)}_{\ell\ell}=n^{-1}$, so
maximizing over $i$ and $\ell$ gives $\delta_{k+1}\le\delta_k$ for all
$k\ge0$, which proves part~(i).

The anchor does more than keep the error from growing. For $k\ge1$ the
anchored entry satisfies $W^{(k)}_{\ell\ell}=n^{-1}$ exactly, so the $j=\ell$
term in \eqref{eq:pulm-column-convex} vanishes and
\begin{equation}\label{eq:pulm-column-error-recursion}
e_\ell^{(k+1)}=A^{(k)}_{-\ell,-\ell}\,e_\ell^{(k)},\qquad k\ge1 .
\end{equation}
The recursion starts at $k=1$ rather than $k=0$ because $W^{(0)}=I_n$ has
$W^{(0)}_{\ell\ell}=1\neq n^{-1}$; the first round is covered
by~\eqref{eq:pulm-column-convex}, which gives
$\|e_\ell^{(1)}\|_\infty\le\delta_0$. Iterating
\eqref{eq:pulm-column-error-recursion} yields
\begin{equation}\label{eq:pulm-column-error-product}
e_\ell^{(K)}=\Big(\prod_{r=1}^{K-1}A^{(r)}_{-\ell,-\ell}\Big)e_\ell^{(1)},
\qquad K\ge1 .
\end{equation}
This is the algebraic effect of the anchor: the error in column $\ell$ evolves
under the principal submatrices $A^{(k)}_{-\ell,-\ell}$, which are nonnegative
with row sums $1-A^{(k)}_{i\ell}\le1$. Whatever weight row $i$ places on the
anchored node $\ell$ is placed on an entry with zero error and is lost from the
error dynamics.

\emph{Step 2: Over every $B$-step block, at least $\eta$ of the weight is lost.}
Let $P_k:=\prod_{r=k}^{k+B-1}A^{(r)}$. By Proposition~\ref{prop:positive},
every entry of $P_k$ is at least $\eta$. Since $P_k$ is row-stochastic,
deleting column $\ell$ removes at least $\eta$ from each row, so every row sum
of $(P_k)_{-\ell,-\ell}$ is at most $1-\eta$. Lemma~\ref{lem:submatrix-product}
gives $0\le\prod_{r=k}^{k+B-1}A^{(r)}_{-\ell,-\ell}\le(P_k)_{-\ell,-\ell}$
entrywise, and for nonnegative matrices entrywise domination implies
domination of row sums. For a matrix, let $\|\cdot\|_\infty$ denote the
operator norm induced by the vector $\infty$-norm, which equals the maximum
absolute row sum. We conclude that
\begin{equation}\label{eq:pulm-block-contraction}
\Big\|\prod_{r=k}^{k+B-1}A^{(r)}_{-\ell,-\ell}\Big\|_\infty
\le\big\|(P_k)_{-\ell,-\ell}\big\|_\infty\le1-\eta,\qquad k\ge0 ,
\end{equation}
while any single factor satisfies $\|A^{(r)}_{-\ell,-\ell}\|_\infty\le1$
because its row sums are at most one.

\emph{Step 3: Assembling the rate.}
Let $K\ge1$. The product in \eqref{eq:pulm-column-error-product} has $K-1$
factors; we group them into $\lfloor(K-1)/B\rfloor$ complete blocks of $B$
consecutive factors, followed by fewer than $B$ remaining factors. Applying
\eqref{eq:pulm-block-contraction} to each block and the bound
$\|A^{(r)}_{-\ell,-\ell}\|_\infty\le1$ to each remaining factor, we obtain
\[
\|e_\ell^{(K)}\|_\infty
\le(1-\eta)^{\lfloor(K-1)/B\rfloor}\|e_\ell^{(1)}\|_\infty
\le(1-\eta)^{\lfloor(K-1)/B\rfloor}\delta_0 ,
\qquad \delta_0=1-\tfrac1n ,
\]
where the value of $\delta_0$ follows from $W^{(0)}=I_n$. The anchored entry
of every column has zero error, so $W^{(K)}-n^{-1}\one\one^\top$ has at most
$n(n-1)$ nonzero entries, each bounded in absolute value by
$\max_\ell\|e_\ell^{(K)}\|_\infty$. Hence
\[
\|W^{(K)}-n^{-1}\one\one^\top\|_F
\le\sqrt{n(n-1)}\,\max_{\ell\in[n]}\|e_\ell^{(K)}\|_\infty
\le(n-1)(1-\eta)^{\lfloor(K-1)/B\rfloor}.
\]
Finally, $\lfloor(K-1)/B\rfloor\ge K/B-2$, and $n\eta\le1$ because each row of
$P_k$ has $n$ entries of size at least $\eta$ that sum to one; thus
$n-1\le n(1-\eta)$ and
\[
\|W^{(K)}-n^{-1}\one\one^\top\|_F
\le\frac{n-1}{(1-\eta)^{2}}(1-\eta)^{K/B}
\le\frac{n}{1-\eta}(1-\eta)^{K/B},
\]
which is \eqref{eq-pulm-convergence-rate} for $K\ge1$. The case $K=0$ holds
because $\|I_n-n^{-1}\one\one^\top\|_F=\sqrt{n-1}\le n/(1-\eta)$.
\end{proof}

{
\begin{remark}[\sc How Anchoring Changes the Convergence Analysis] \label{rmk:anchor-vs-ergodic} Proposition~\ref{prop:row-sto-linear-convergence} and Theorem~\ref{thm:exp convergence} both rely on the uniform positivity in Proposition~\ref{prop:positive}, but establish contraction of different quantities. The classical ergodicity analysis (Lemma~\ref{lem(app):A-limit}) shows that the \emph{spread} of each column of $A^{(k)}\cdots A^{(0)}$, defined as the difference between its largest and smallest entries, contracts by a factor of $1-\eta$ per block. This ensures consensus, but the limiting weights $\pi$ depend on the mixing sequence and can be influenced by future matrices. In contrast, Theorem~\ref{thm:exp convergence} establishes contraction of the \emph{maximum absolute error} $\max_{i,\ell}|W^{(k)}_{i\ell}-n^{-1}|$ relative to a prescribed target. Anchoring makes one entry of each column exact after every adjustment, so the remaining errors follow the substochastic recursion~\eqref{eq:pulm-column-error-recursion}. Its matrix product over each block has row sums at most $1-\eta$, yielding the contraction in~\eqref{eq:pulm-block-contraction}. Thus, anchoring turns convergence toward an unspecified consensus value into convergence toward the prescribed averaging limit, without estimating the network-dependent consensus weights required by Perron-correction methods. \end{remark}}


\subsection{The PULM Method: Decentralized Implementation}
\label{sec-pulm-method}

Recursion~\eqref{eq:gossip-adjust} is sufficient for average consensus under
time-varying row-stochastic mixing matrices: letting $\vz^{(k)} = W^{(k)}\vx$,
Theorem~\ref{thm:exp convergence} gives
$\vz^{(k)} \to n^{-1}\one\one^\top\vx$. It is not, however, decentralized:
each node $i$ can maintain the $i$-th row of $W^{(k)}$ by pulling the rows of its
in-neighbors, but forming $z_i^{(k)}=\sum_{j}W^{(k)}_{ij}x_j$ requires the vectors
$x_j$ of all nodes. To address this, we propose the following recursion on $\vz^{(k)}$:
\begin{equation}\label{eq-z-update}
\hspace{-5mm} \textrm{(PULM)}\quad \quad \quad \vz^{(k+\frac{1}{2})} = A^{(k)}\vz^{(k)},\qquad
\vz^{(k+1)} = \vz^{(k+\frac{1}{2})} - D^{(k)}\vx.
\end{equation}
To verify that~\eqref{eq-z-update} indeed yields $\vz^{(k+1)} = W^{(k+1)}\vx$,
observe that
\[
\begin{aligned}
    \vz^{(k+1)} &= A^{(k)}\vz^{(k)} - D^{(k)}\vx
    \overset{\eqref{eq:gossip-adjust}}{=} A^{(k)}\vz^{(k)} - \big(A^{(k)}W^{(k)} - W^{(k+1)}\big)\vx\\
    &= A^{(k)}\vz^{(k)} - A^{(k)}W^{(k)}\vx + W^{(k+1)}\vx,
\end{aligned}
\]
which implies
\begin{equation}\label{eq:z=Wx}
\vz^{(k+1)} - W^{(k+1)}\vx
= A^{(k)}\big(\vz^{(k)} - W^{(k)}\vx\big)
= \prod_{i=0}^{k} A^{(i)}\big(\vz^{(0)} - W^{(0)}\vx\big) = 0,
\end{equation}
where the last equality follows from the initialization $\vz^{(0)} = \vx$ and
$W^{(0)} = I_n$. A decentralized implementation of~\eqref{eq-z-update} is given
in Algorithm~\ref{alg:PWM}. Each node $i$ maintains, in addition to its own
iterate $z_i^{(k)}$, the distribution vector $w_i^{(k)}$, which stores the $i$-th row vector 
of $W^{(k)}$ and supplies the scalar $d_i^{(k)}$ needed for the adjust step. We
name the resulting method \textsc{PULM} (Pull with Memory), since gossiping with
a row-stochastic matrix is commonly referred to as a ``pull'' operation and the
vector $w_i^{(k)}$ serves as a node-level ``memory'' of the mixing process.

\begin{algorithm}[t]
\caption{Pull with Memory (node-wise implementation of~\eqref{eq:gossip-adjust})}
\label{alg:PWM}
\begin{algorithmic}[1]
\State \textbf{Input:} Vector $x_i\in\mathbb{R}^d$ at each node $i\in[n]$ to be averaged
\For{each node $i\in[n]$ in parallel}
    \State Initialize $z_i^{(0)}=x_i$ and $w_i^{(0)}=\bm{e}_i$, where $\bm{e}_i$ is the $i$-th column of $I_n$
    \For{$k = 0, 1, 2, \ldots, K-1$}
        \State \textbf{Gossip:} $z_i^{(k+\frac{1}{2})} = \sum_{j=1}^n a_{ij}^{(k)} z_j^{(k)}$ and $w_i^{(k+\frac{1}{2})} = \sum_{j=1}^n a_{ij}^{(k)} w_j^{(k)}$
        \State \textbf{Adjust:} $d_i^{(k)} = [w_i^{(k+\frac{1}{2})}]_i - \frac{1}{n}$
        \State \textbf{Update:} $z_i^{(k+1)} = z_i^{(k+\frac{1}{2})} - d_i^{(k)} x_i$ and $w_i^{(k+1)} = w_i^{(k+\frac{1}{2})} - d_i^{(k)} \bm{e}_i$
    \EndFor
\EndFor
\State \textbf{Output:} $z_i^{(K)}$ at each node $i\in[n]$
\end{algorithmic}
\end{algorithm}

Combining~\eqref{eq:z=Wx} with Theorem~\ref{thm:exp convergence}(ii) gives the
convergence rate of \textsc{PULM}: the output of Algorithm~\ref{alg:PWM}
satisfies
\begin{equation}\label{eq:pulm-z-rate}
\|\vz^{(K)}-n^{-1}\one\one^\top\vx\|_F
\le \frac{n}{1-\eta}(1-\eta)^{K/B}\|\vx\|_F .
\end{equation}

\noindent \textbf{Communication and Storage.} Each node stores $x_i,z_i^{(k)}\in\mathbb{R}^d$ and $w_i^{(k)}\in\mathbb{R}^n$, where $w_i^{(k)}$ is the $i$-th row of $W^{(k)}$. Hence, the per-node storage is $2d+n=\mathcal{O}(d+n)$ scalars. In each round, the pair $(z_i^{(k)},w_i^{(k)})$ is jointly communicated, resulting in $d+n$ scalars per message; the adjustment step is local and requires no additional communication.



\section{Decentralized Optimization over TVBNs}\label{sec:PWM-GD}
While PULM is effective for achieving average consensus with time-varying row-stochastic matrices, it cannot directly solve the optimization problem~\eqref{eq(main):f definition}. In this section, we extend PULM to address this setting.

\vspace{1mm}
\subsection{Main Recursion} 
We propose a two-loop method for solving problem~\eqref{eq(main):f definition} using time-varying row-stochastic matrices alone. At outer iteration $k$, the inner loop performs $R_k$ rounds of updates. It is initialized with $W^{(k,0)} = I_n$ and $\vz^{(k,0)} = \vx^{(k)} - \gamma \vg^{(k)}$, and proceeds, for inner loop $r = 0,1,\ldots,R_k-1$, as
\begin{subequations}\label{eq:pulm_dgd_main}
\renewcommand{\theequation}{\theparentequation.\arabic{equation}}
\begin{align}
W^{(k,r+\frac{1}{2})} &= A^{(k,r)}W^{(k,r)}, \label{eq:pulm-dgd-1} \\
D^{(k,r)} &= \mathrm{Diag}\big(W^{(k,r+\frac{1}{2})}\big)- n^{-1}I_n, \label{eq:pulm-dgd-2}\\
W^{(k,r+1)} &= W^{(k,r+\frac{1}{2})}-D^{(k,r)}, \label{eq:pulm-dgd-3}\\
\vz^{(k,r+1)} &= A^{(k,r)}\vz^{(k,r)}+\gamma D^{(k,r)}\vg^{(k)}. \label{eq:pulm-dgd-4}
\end{align}
\end{subequations}
The outer iterate is then updated as $\vx^{(k+1)} = \vz^{(k,R_k)}$, which starts the next outer iteration $k+1$. Note that \eqref{eq:pulm-dgd-1}--\eqref{eq:pulm-dgd-3} coincide with the PULM recursion~\eqref{eq:gossip-adjust}, which implies that $W^{(k,r)} \to n^{-1}\one\one^\top$ exponentially as $r$ grows. The only new update is~\eqref{eq:pulm-dgd-4}, which feeds the correction $D^{(k,r)}$ back into $\vz^{(k,r+1)}$. A decentralized implementation, termed PULM-DGD, is presented in Algorithm~\ref{alg:PWM-GD}. Note that each outer iteration computes the local gradients once and performs $R_k$ communication rounds, during which the gradients remain fixed while the row-stochastic mixing matrix $A^{(k,r)}$ may vary from round to round. The inner rounds are consecutive communication rounds of the network: with $t_k:=\sum_{j<k}R_j$, we write $A^{(k,r)}:=A^{(t_k+r)}$ for all $r\ge0$, where $\{A^{(t)}\}_{t\ge0}$ is the sequence in Assumptions~\ref{ass:tv graph}--\ref{ass:matrix}. Since these assumptions and Proposition~\ref{prop:positive} hold uniformly in the starting time, Theorem~\ref{thm:exp convergence} and Proposition~\ref{prop:row-sto-linear-convergence} apply to every sequence $\{A^{(k,r)}\}_{r\ge0}$.

\vspace{1mm}
{
\noindent \textbf{Communication and Storage.} According to Algorithm~\ref{alg:PWM-GD}, node $i$ stores the local gradient
$g_i^{(k)}\in\mathbb{R}^d$, which is fixed throughout the inner loop, together with
$z_i^{(k,r)}\in\mathbb{R}^d$ and $w_i^{(k,r)}\in\mathbb{R}^n$; since
$x_i^{(k+1)}=z_i^{(k,R_k)}$, the variable $x_i^{(k)}$ may be overwritten once $g_i^{(k)}$ and
$z_i^{(k,0)}$ are formed. The per-node storage is thus $2d+n=\mathcal{O}(d+n)$ scalars and does
not accumulate over the inner loop, while each inner round communicates the pair
$(z_i^{(k,r)},w_i^{(k,r)})$, i.e., $d+n$ scalars, with the adjust step performed locally. The
$\mathcal{O}(n)$ component is not specific to PULM: Pull-Diag~\cite{mai2016distributed,xi2018linear,liangachieving}, FRSD~\cite{ghaderyan2023fast}, D-DNGT~\cite{lu2020nesterov} and
FROST~\cite{FROST-Xinran} all maintain and transmit an $n$-dimensional vector as well, in their case to
estimate the Perron vector of the mixing matrix, so all these methods share the same
$\mathcal{O}(d+n)$ order in both per-round communication and per-node storage. The distinction
lies elsewhere: Pull-Diag, FRSD, D-DNGT and FROST estimate a Perron vector that is well defined only for a
static mixing matrix, whereas the $\mathcal{O}(n)$ memory of PULM-DGD carries no such dependence,
since $w_i^{(k,r)}$ records how the mixing has actually proceeded rather than a property of a
fixed network. PULM-DGD therefore extends to time-varying row-stochastic networks at the same
order of communication and storage that existing methods require on static ones, see Table~\ref{tab:overhead}.}

\subsection{Insights behind PULM-DGD}

\begin{algorithm}[t]
\caption{PULM-based Decentralized Gradient Descent (PULM-DGD)}
\label{alg:PWM-GD}
\begin{algorithmic}[1]
\State \textbf{Input:} Step size $\gamma > 0$, inner iterations $\{R_k\}_{k\ge 0}$
\For{each node $i\in[n]$ in parallel}
    \State Initialize all variables to the same arbitrary value $x_i^{(0)} = x^{(0)}$
    \For{$k = 0, 1, 2, \ldots, K-1$}
        \State Compute local gradient: $g_i^{(k)} = \nabla f_i(x_i^{(k)})$
        \State Initialize $z_i^{(k,0)} = x_i^{(k)}- \gamma g_i^{(k)}$ and $w_i^{(k,0)} = \bm{e}_i$
        \For{$r=0,1,2,\ldots, R_k-1$} \Comment{{\small {\color{black}PULM inner loop}}}
            \State \textbf{Gossip:} $z_i^{(k,r+\frac{1}{2})} = \sum_{j=1}^n a_{ij}^{(k,r)} z_j^{(k,r)}$, $w_i^{(k,r+\frac{1}{2})} = \sum_{j=1}^n a_{ij}^{(k,r)} w_j^{(k,r)}$
            \State \textbf{Adjust:} $d_i^{(k,r)} = [w_i^{(k,r+\frac{1}{2})}]_i - \frac{1}{n}$
            \State \textbf{Update:} $z_i^{(k,r+1)} = z_i^{(k,r+\frac{1}{2})} +\gamma   d_i^{(k,r)} g_i^{(k)}$,  $w_i^{(k,r+1)} = w_i^{(k,r+\frac{1}{2})} - d_i^{(k,r)} \bm{e}_i$
        \EndFor
        \State $x_i^{(k+1)} = z_i^{(k,R_k)}$
    \EndFor
\EndFor
\State \textbf{Output:} $x_i^{(K)}$ at each node $i\in[n]$
\end{algorithmic}
\end{algorithm}

We now show the insight behind the above procedure. Observe in \eqref{eq:pulm_dgd_main} that
\[
\begin{aligned}
    \vz^{(k,r+1)} &= A^{(k,r)}\vz^{(k,r)} + \gamma \bigl(A^{(k,r)}W^{(k,r)} - W^{(k,r+1)}\bigr)\vg^{(k)} \\
    &= A^{(k,r)}\vz^{(k,r)} + \gamma A^{(k,r)}W^{(k,r)} \vg^{(k)} - \gamma W^{(k,r+1)} \vg^{(k)},
\end{aligned}
\]
which implies
\[
\vz^{(k,r+1)} + \gamma W^{(k,r+1)}\vg^{(k)} = A^{(k,r)} \bigl(\vz^{(k,r)} + \gamma W^{(k,r)}\vg^{(k)}\bigr).
\]
Applying this relation recursively, we obtain
\[
\vz^{(k,R_k)} + \gamma W^{(k,R_k)} \vg^{(k)} = \prod_{r=0}^{R_k-1} A^{(k,r)} \bigl(\vz^{(k,0)} + \gamma W^{(k,0)}\vg^{(k)}\bigr) = \prod_{r=0}^{R_k-1} A^{(k,r)} \vx^{(k)},
\]
where the last equality follows from $W^{(k,0)} = I_n$ and $\vz^{(k,0)} = \vx^{(k)} - \gamma \vg^{(k)}$. Thus, 
\begin{align}\label{eq:PULM-DGD-iter}
    \vx^{(k+1)} &= \vz^{(k,R_k)} = \prod_{r=0}^{R_k-1} A^{(k,r)} \vx^{(k)} - \gamma W^{(k,R_k)} \vg^{(k)} \nonumber \\
    &=
    \mathrm{Gossip}\big(\vx^{(k)}, R_k\big)
    \;-\;
    \gamma\,\mathrm{PULM}\big(\vg^{(k)}, R_k\big),
\end{align}
where
$
\mathrm{Gossip}(\vx^{(k)}, R_k)
\;:=\;
\big(\prod_{r=0}^{R_k-1} A^{(k,r)}\big)\,\vx^{(k)}
$. Hence, each outer iteration is equivalent to applying standard gossip to the parameters and PULM to the gradients, both using the same sequence of $R_k$ row-stochastic mixing matrices. The gossip component drives the local parameters toward consensus, while the PULM component progressively removes the bias in gradient aggregation and drives the mixed gradients toward the global average. These two mechanisms will ensure that PULM-DGD finds an $\epsilon$-stationary solution of problem~\eqref{eq(main):f definition} for any $\epsilon>0$.

\begin{table}[t]
\centering
\caption{Per-node storage, per-round communication, and convergence rate of
decentralized methods relying on row-stochastic mixing alone. Here $d$ is the
problem dimension and $n$ is the number of nodes. Rate refers to the
deterministic nonconvex convergence rate in terms of the number of communication
rounds $T$; ``--'' indicates that no such rate has been established.}
\label{tab:overhead}
\small
\begin{tabular}{@{}lcccc@{}}
\toprule
Method & Storage & Communication & Row-stochasticity  & Nonconvex rate \\
\midrule
Xi-row~\cite{xi2018linear}       & $3d+n$ & $2d+n$ & Static       & -- \\
FROST~\cite{FROST-Xinran}        & $3d+n$ & $2d+n$ & Static       & -- \\
FROZEN~\cite{xin2019distributed} & $4d+n$ & $2d+n$ & Static       & -- \\
D-DNGT~\cite{lu2020nesterov}     & $5d+n$ & $2d+n$ & Static       & -- \\
FRSD~\cite{ghaderyan2023fast}    & $2d+n$ & $d+n$  & Static       & -- \\
Pull-Diag-GT~\cite{liangachieving} & $3d+n$ & $2d+n$ & Static     & $\mathcal{O}(1/T)$ \\
PULM-DGD (ours)                  & $2d+n$ & $d+n$  & Time-varying & $\mathcal{O}(\ln T/T)$ \\
\bottomrule
\end{tabular}
\vspace{-5mm}
\end{table}

\subsection{Convergence Analysis} 

To facilitate the convergence analysis, we introduce the following notation. Let $\bar{g}^{(k)} := \frac{1}{n}\sum_{i=1}^n g_i^{(k)}$ and $\bar{x}^{(k)} := \frac{1}{n}\sum_{i=1}^n x_i^{(k)}$ denote the average gradient and average iterate, respectively. Let $\Delta_x^{(k)} := \vx^{(k)} - \fulcon \vx^{(k)}$ denote the consensus error matrix. For notational convenience, we define

\[
C_A := n^{3/2}\eta^{-1} \cdot \frac{1+\eta}{1-\eta}, \qquad C_W := \frac{n}{1-\eta}, \qquad \beta := \sqrt[B]{1-\eta},
\]
which appear in~\eqref{eq:A-linear-convergence} and~\eqref{eq-pulm-convergence-rate}, respectively. We  denote the product of row-stochastic matrices $\prod_{r=0}^{R_k-1} A^{(k,r)}$ by $A^{(k)}$, and write $W^{(k)}$ for $W^{(k,R_k)}$ (in this section, $A^{(k)}$ and $W^{(k)}$ denote these quantities, not the single mixing matrix and PULM iterate of Section~\ref{sec:prelim}). Finally, when there is no ambiguity, a row vector resulting from vector-matrix multiplication (e.g., $\one^\top \vx^{(k)}$) is treated as a column vector when added to or subtracted from other column vectors.

As is standard in the analysis of decentralized gradient descent, the key step is to establish a descent property with respect to the consensus quantities $\bar{x}^{(k)}$ and $\bar{g}^{(k)}$. This is formalized in Lemma~\ref{lemma:descent}.
\begin{lemma}\label{lemma:descent}
    Suppose Assumptions~\ref{ass:tv graph}--\ref{ass:smooth} hold for Algorithm~\ref{alg:PWM-GD}. When \(\gamma\le \frac{1}{6L}\) and \(R_k\ge \ln(C_W)/\ln(1/\beta)\), for any \(k\ge 0\), we have
    \begin{equation}
    \begin{aligned}
        \frac{\gamma}{4}\|\bar{g}^{(k)}\|^2+\frac{\gamma}{6}\|\nabla f(\bar{x}^{(k)})\|^2 &\le (1+8\gamma LC_W^2\beta^{2R_k})\big(f(\bar{x}^{(k)})-f^*\big)\\
        &\quad -\big(f(\bar{x}^{(k+1)})-f^*\big)+\frac{2}{\gamma}\|\Delta_x^{(k)}\|_F^2,
    \end{aligned}
    \label{eq:descent-lemma}
    \end{equation}
    where \(f^*:=n^{-1}\sum_{i=1}^n\inf_{x\in\mathbb{R}^d}f_i(x)\).
\end{lemma}
\begin{proof}{
Consider the update of $\bar{x}^{(k+1)}$ using \eqref{eq:PULM-DGD-iter}:
\begin{align*}
    \bar{x}^{(k+1)}&=n^{-1}\one^\top\big( A^{(k)}\vx^{(k)}-\gamma W^{(k)}\vg^{(k)}\big)\\
    &=\bar{x}^{(k)}-\gamma\bar{g}^{(k)}+\underbrace{n^{-1}\one^\top (A^{(k)}-E_n)\Delta_x^{(k)}}_{\text{Consensus Error}}-\underbrace{n^{-1}\gamma \one^\top (W^{(k)}-\fulcon)\vg^{(k)}}_{\text{Descent Deviation}}
\end{align*}
We apply the $L$-smooth inequality on $\bar{x}^{(k+1)}$, $\bar{x}^{(k)}$ and using average inequality:
\begin{equation}
\begin{aligned}
    f(\bar{x}^{(k+1)})&\le f(\bar{x}^{(k)}) -\gamma \ip{\bar{g}^{(k)},\nabla f(\bar{x}^{(k)})}\\
    &\quad +n^{-1}\ip{\one^\top(A^{(k)}-\fulcon)\Delta_x^{(k)},\nabla f(\bar{x}^{(k)})}\\
    &\quad-n^{-1}\gamma \ip{\one^\top(W^{(k)}-\fulcon)\vg^{(k)},\nabla f(\bar{x}^{(k)})}+\frac{3\gamma^2 L}{2}\norm{\bar{g}^{(k)}}^2\\
    &\quad +\frac{3L}{2n^2}\norm{\one^\top(A^{(k)}-\fulcon)\Delta_x^{(k)}}^2+\frac{3\gamma^2 L}{2n^2}\norm{\one^\top(W^{(k)}-\fulcon)\vg^{(k)}}^2
\end{aligned}
\label{ieq:des-1}
\end{equation}
To proceed on, we split the right-hand side into 5 terms, which are
\begin{align*}
    &\Delta_1^{(k)}:=-\gamma \ip{\bar{g}^{(k)},\nabla f(\bar{x}^{(k)})}+\frac{3\gamma^2 L}{2}\norm{\bar{g}^{(k)}}^2, \\
    & \Delta_2^{(k)}:=n^{-1}\ip{\one^\top(A^{(k)}-\fulcon)\Delta_x^{(k)},\nabla f(\bar{x}^{(k)})},\\
    &\Delta_3^{(k)}:=-n^{-1}\gamma \ip{\one^\top(W^{(k)}-\fulcon)\vg^{(k)},\nabla f(\bar{x}^{(k)})},\\ &\Delta_4^{(k)}:=\frac{3L}{2n^2}\norm{\one^\top(A^{(k)}-\fulcon)\Delta_x^{(k)}}^2,\quad \Delta_5^{(k)}:=\frac{3\gamma^2 L}{2n^2}\norm{\one^\top(W^{(k)}-\fulcon)\vg^{(k)}}^2.
\end{align*}
 When $\gamma \le \frac{1}{6L}$, the first term can be bounded as:
\begin{align*}
     \Delta_1^{(k)}
    &= -\frac{\gamma-3\gamma^2 L}{2}\norm{\bar{g}^{(k)}}^2 -\frac{\gamma}{2}\norm{\nabla f(\bar{x}^{(k)})}^2+\frac{\gamma}{2}\norm{\bar{g}^{(k)}-\nabla f(\bar{x}^{(k)})}^2 \\
    &\le -\frac{\gamma}{4}\norm{\bar{g}^{(k)}}^2 -\frac{\gamma}{2}\norm{\nabla f(\bar{x}^{(k)})}^2+\frac{\gamma L^2}{2n}\norm{\Delta_x^{(k)}}_F^2
\end{align*}

For the second term, we have the observation: For any $y\in\mathbb{R}^n$, $\bar{y}:=n^{-1}\one^\top y$ represents the average of all its elements, while $\tilde{y}:=n^{-1}\one^\top A^{(k)}y$ is a convex but not necessarily average combination of all its elements. But the following inequality always holds: $(\tilde{y}-\bar{y})^2\le \max_{i}|y_i-\bar{y}|^2\le \sum_{i=1}^n (y_i-\bar{y})^2$. Generalizing to $d$ columns, we have $\|n^{-1}\one^\top (A^{(k)}-\fulcon)\vy\|^2\le \|\vy-\fulcon\vy\|_F^2$. And we further have $\Delta_x^{(k)}-\fulcon\Delta_x^{(k)}=\Delta_x^{(k)}$. Hence, using Cauchy-Schwarz inequality, $\Delta_2^{(k)}$ can be bounded as:
\begin{align*}
      \Delta_2^{(k)}
    &\le \norm{n^{-1}\one^\top(A^{(k)}-\fulcon)\Delta_x^{(k)}}\cdot \norm{\nabla f(\bar{x}^{(k)})}\\
    &\le \norm{\Delta_x^{(k)}}_F\cdot \norm{\nabla f(\bar{x}^{(k)})}\le \frac{p_k}{2}\norm{\Delta_x^{(k)}}_F^2+\frac{1}{2p_k}\norm{\nabla f(\bar{x}^{(k)})}^2,
\end{align*}
where the parameter $p_k$ is to be determined. 

Similarly, the third term can be bounded as:
\begin{align*}
    \Delta_3^{(k)}
    &\le \gamma \norm{n^{-1}\one^\top(W^{(k)}-\fulcon)\vg^{(k)}}\cdot \norm{\nabla f(\bar{x}^{(k)})}\\
    &\le \gamma n^{-1/2} \|W^{(k)}-\fulcon\|_F \|\vg^{(k)}\|_F\|\nabla f(\bar{x}^{(k)})\|\\
    &\le \gamma  \frac{C_W \beta^{R_k}}{\sqrt{n}}(\frac{q_k}{2}\norm{\vg^{(k)}}_F^2+\frac{1}{2q_k}\norm{\nabla f(\bar{x}^{(k)})}^2),
\end{align*}
where the third inequality uses \eqref{eq-pulm-convergence-rate}, and $q_k$ is a constant to be determined. 

The fourth term is smaller than \(\frac{3L}{2}\norm{\Delta_x^{(k)}}^2\), and the fifth term is smaller than $\frac{3 C_W^2\gamma^2L \beta^{2R_k}}{2n}\norm{\vg^{(k)}}_F^2$. Now combine these estimates and plug them into \eqref{ieq:des-1}, we obtain 
\begin{align*}
    &\quad\frac{\gamma}{4}\norm{\bar{g}^{(k)}}^2+(\frac{\gamma}{2}-\frac{ 1}{2p_k}-\frac{ C_W \beta^{R_k}\gamma}{2q_k\sqrt{n}})\norm{\nabla f(\bar{x}^{(k)})}^2\\
    &\le f(\bar{x}^{(k)})-f(\bar{x}^{(k+1)})+ (\frac{\gamma L^2}{2n}+\frac{ p_k}{2}+\frac{3L}{2})\norm{\Delta_x^{(k)}}_F^2 \\
    &\quad+ (\frac{\gamma  C_W \beta^{R_k}q_k}{2\sqrt{n}}+\frac{3 C_W^2\gamma^2 L  \beta^{2R_k}}{2n})\norm{\vg^{(k)}}_F^2.
\end{align*}
Selecting $q_k=\frac{3 C_W \beta^{R_k}}{\sqrt{n}}$, $p_k=\frac{3}{\gamma}$, we obtain that 
\begin{align*}
    &\quad\frac{\gamma}{4}\norm{\bar{g}^{(k)}}^2+\frac{\gamma}{6}\norm{\nabla f(\bar{x}^{(k)})}^2\\
    &\le f(\bar{x}^{(k)})-f(\bar{x}^{(k+1)})+ (\frac{\gamma L^2}{2n}+\frac{3}{2\gamma}+\frac{3L}{2})\norm{\Delta_x^{(k)}}_F^2 + \frac{3\gamma (1+\gamma L)C_W^2 \beta^{2R_k}}{2n}\norm{\vg^{(k)}}_F^2.
\end{align*}
By taking $\gamma\le \frac{1}{6L}$, we obtain that 
\begin{equation}
\begin{aligned}
    &\quad\frac{\gamma}{4}\norm{\bar{g}^{(k)}}^2+\frac{\gamma}{6}\norm{\nabla f(\bar{x}^{(k)})}^2\\
    &\le f(\bar{x}^{(k)})-f(\bar{x}^{(k+1)}) + \frac{127}{72\gamma}\|\Delta_x^{(k)}\|_F^2 + \frac{2\gamma C_W^2\beta^{2R_k} }{n}\norm{\vg^{(k)}}_F^2.
\end{aligned}
\label{eq:descent-1}
\end{equation}

Finally, to deal with the stacked gradient term $\|\vg^{(k)}\|_F^2$, with $L$-smooth assumption, we know that $\forall x, y\in \mathbb{R}^d, i\in [n]$,
\[f_i( {y})\le f_i( {x})+\ip{\nabla f_i( {x}), {y}- {x}}+\frac{L}{2}\norm{ {y}- {x}}^2.\]
By taking $ {y}= {x}-\frac{1}{L}\nabla f_i( {x})$, we obtain that $\frac{1}{2L}\norm{\nabla f_i( {x})}^2 \le f_i( {x})-f_i( {y})\le f_i( {x})-f_i^*$. Furthermore, using $L$-smoothness property and Cauchy-Schwarz inequality, we have
\begin{align}
    \|\vg^{(k)}\|_F^2 &=\norm{\nabla F(\vx^{(k)})}_F^2 \le 2\norm{\nabla F(\vx^{(k)})-\nabla F(\fulcon \vx^{(k)})}_F^2 + 2\norm{\nabla F(\fulcon \vx^{(k)})}_F^2\nonumber\\
    &= 2\sum_{i=1}^n \|\nabla f_i(x_i^{(k)})-\nabla f_i(\bar{x}^{(k)})\|^2 + 2\sum_{i=1}^n \norm{\nabla f_i( {\bar{x}}^{(k)})}^2\nonumber\\
    &\le 2L^2\sum_{i=1}^n \|x_i^{(k)}-\bar{x}^{(k)}\|^2+4L\sum_{i=1}^n(f_i( {\bar{x}}^{(k)})-f_i^*)\nonumber\\
    &=2L^2\norm{\Delta_x^{(k)}}_F^2+4nL(f( {\bar{x}}^{(k)})-f^*),\label{eq:vg-bound}
\end{align}
where $f^*:=n^{-1}\sum_{i=1}^n f_i^*$. To further simplify the inequality, we require $R_k\ge \ln(C_W)/\ln(1/\beta)$ to ensure $C_W\beta^{R_k}\le 1$. Substituting $\|\vg^{(k)}\|_F^2$ with \eqref{eq:vg-bound} into \eqref{eq:descent-1} and using $\gamma\le \frac{1}{6L}$, we obtain the conclusion.}
\end{proof}

Beyond the optimality of $\bar{x}^{(k)}$, achieving consensus is equally important. To this end, we analyze the evolution of the consensus error matrix $\Delta_x^{(k)}$ in Lemma~\ref{lemma:consensus-iteration}.
\begin{lemma}\label{lemma:consensus-iteration}
    Suppose Assumptions~\ref{ass:tv graph}--\ref{ass:smooth} hold for Algorithm~\ref{alg:PWM-GD}. When $\gamma\le \frac{1}{6L}$ and $R_k\ge \max \{ \ln(C_W)/\ln(1/\beta)$, $ \ln(3C_A)/\ln(1/\beta)\}$, we have
    \begin{equation}
        \|\Delta_x^{(k+1)}\|_F^2\le \frac{1}{3}\|\Delta_x^{(k)}\|_F^2+8n\gamma^2 L C_W^2\beta^{2R_k}\big(f(\bar{x}^{(k)})-f^*\big),
    \label{eq:lemma-consensus-iteration}
    \end{equation}
    where $f^*$ is defined the same as in Lemma~\ref{lemma:descent}.
\end{lemma}
\begin{proof}
{
Recall the update format \eqref{eq:PULM-DGD-iter} with simplified notations:
\begin{equation}\label{eq:con-1}
    \vx^{(k+1)} = A^{(k)}\vx^{(k)}-\gamma W^{(k)}\vg^{(k)}.
\end{equation}
Let $A^{(k)}_\infty:=\lim_{s\to\infty}A^{(k,s)}\cdots A^{(k,1)}A^{(k,0)}=\one(\pi^{(k)})^\top$ be the limit of the product of all mixing matrices from round $t_k$ onward, which exists by Proposition~\ref{prop:row-sto-linear-convergence} applied to $\{A^{(k,r)}\}_{r\ge0}$; it is an analysis device only and is never computed by the algorithm. By~\eqref{eq:A-linear-convergence}, $\|A^{(k)}-A^{(k)}_\infty\|_F\le C_A\beta^{R_k}$. Left-multiply $(I-\fulcon)$ on both sides of \eqref{eq:con-1} and we obtain that
\begin{align}\label{eq:con-2}
    \Delta_x^{(k+1)} = (I-\fulcon)(A^{(k)}-A_\infty^{(k)})\Delta_x^{(k)}-\gamma (I-\fulcon)(W^{(k)}-\fulcon)\vg^{(k)}.
\end{align}
Therefore, we can apply the Cauchy-Schwarz inequality on~\eqref{eq:con-2} and obtain that
\begin{align*}
    \norm{\Delta_x^{(k+1)}}_F^2 &\le 2\|(I-\fulcon)(A^{(k)}-A_\infty^{(k)})\Delta_x^{(k)}\|_F^2+2\gamma^2\|(I-\fulcon)(W^{(k)}-\fulcon)\vg^{(k)}\|_F^2\\
    &\le 2\|A^{(k)}-A_\infty^{(k)}\|_F^2\|\Delta_x^{(k)}\|_F^2+2\gamma^2 \|W^{(k)}-\fulcon\|_F^2\|\vg^{(k)}\|_F^2\\
    &\le 2C_A^2\beta^{2R_k}\|\Delta_x^{(k)}\|_F^2+2\gamma^2 C_W^2\beta^{2R_k}\|\vg^{(k)}\|_F^2\\
    &\le (2C_A^2\beta^{2R_k} + 4 \gamma^2L^2C_W^2 \beta^{2R_k})\|\Delta_x^{(k)}\|_F^2 + 8n\gamma^2 L C_W^2 \beta^{2R_k}(f(\bar{x}^{(k)})-f^*),
\end{align*}
where the second inequality is due to $\|I_n-\fulcon\|_2=1$, the third inequality uses the linear convergence rate in \eqref{eq:A-linear-convergence} and \eqref{eq-pulm-convergence-rate}, and the last inequality is from \eqref{eq:vg-bound}. Choose $R_k\ge \max\{\ln(C_W)/\ln(1/\beta),\ln(3C_A)/\ln(1/\beta)\}$ and $\gamma \le \frac{1}{6L}$ to bound the coefficient of $\|\Delta_x^{(k)}\|_F^2$ and the conclusion follows.}
\end{proof}

Next, multiplying both sides of~\eqref{eq:lemma-consensus-iteration} by $\frac{12}{\gamma}$ and adding the result to~\eqref{eq:descent-lemma} yields
\begin{equation}
\begin{aligned}
    &\quad \frac{\gamma}{4}\norm{\bar{g}^{(k)}}^2+\frac{\gamma}{6}\norm{\nabla f(\bar{x}^{(k)})}^2 + \frac{12}{\gamma} \|\Delta_x^{(k+1)}\|_F^2 \\
    &\le \big(1+ (8+96n)\gamma L C_W^2\beta^{2R_k}\big)(f(\bar{x}^{(k)})-f^*)-(f(\bar{x}^{(k+1)})-f^*)+\frac{6}{\gamma}\norm{\Delta_x^{(k)}}_F^2.
\end{aligned}
\label{eq(app):consensus-1}
\end{equation}
Since the coefficients of $f(\bar{x}^{(k)}) - f^*$ and $f(\bar{x}^{(k+1)}) - f^*$ differ, we require the following lemma to account for the resulting discrepancy.
\begin{lemma}\label{lemma:absorbing-2}
    Let $\{S_k\}_{k=0}^{K-1}$, $\{D_k\}_{k=0}^{K}$, and $\{F_k\}_{k=0}^{K-1}$ be three sequences of nonnegative values, and let $c > 0$ be a constant. If the inequality
    \begin{equation}\label{eq:absorb-provided}
        S_k \le \left(1 + \frac{c}{(k+1)^2}\right) D_k - D_{k+1} + F_k
    \end{equation}
    holds for all $0 \le k \le K-1$, then
\begin{equation}\label{eq:absorbing-4}
        \sum_{k=0}^{K-1} S_k \le e^{2c} D_0 + \frac{e^{2c}}{1+c} \sum_{k=0}^{K-1} F_k.
    \end{equation}
\end{lemma}
\begin{proof}{
Define $u_k:=1+\frac{c}{(1+k)^2}$, $U_k=\prod_{l=0}^{k}u_l$, and $U_{-1}:=1$ for convenience, which imply $u_k=U_k/U_{k-1}, \forall k\ge 0$. Then \eqref{eq:absorb-provided} can be transformed as 
\begin{equation}\label{eq:absorbing-u}
    U_{k}^{-1}S_k\le U_{k-1}^{-1}D_k-U_k^{-1}D_{k+1}+U_k^{-1}F_k,\quad \forall k\ge 0.
\end{equation}
Since $u_k\ge 1,\forall k\ge 0$, $U_k$ is monotonically increasing. Therefore, summing \eqref{eq:absorbing-u} from $k=0$ to $K-1$ and telescoping, we have
\begin{equation*}
U_{K-1}^{-1}\sum_{k=0}^{K-1}S_k\le \sum_{k=0}^{K-1}U_k^{-1}S_k\le D_0-U_{K-1}^{-1}D_K+\sum_{k=0}^{K-1} U_k^{-1}F_k\le D_0 + U_0^{-1}\sum_{k=0}^{K-1}F_k.
\end{equation*}
On one hand, $U_0=u_0=1+c$. On the other hand,
\[
\begin{aligned}
U_{K-1} &= \prod_{k=0}^{K-1}\Big(1+\frac{c}{(1+k)^2}\Big)=\exp\bigg(\sum_{k=0}^{K-1}\ln \Big(1+\frac{c}{(1+k)^2}\Big)\bigg)\\
&\le \exp\bigg(\sum_{k=0}^{K-1}\frac{c}{(1+k)^2}\bigg)\le \exp(2c).
\end{aligned}
\]
Substitute $U_0$ with its value and $U_{K-1}$ with its upper bound and we have the conclusion.}
\end{proof}

To apply Lemma~\ref{lemma:absorbing-2}, we require $(8 + 96n)\gamma L C_W^2 \beta^{2R_k} \le c/(1+k)^2$. A sufficient condition is $\gamma \le (16 \cdot (1 + 12n) L C_W^2)^{-1}$ and $R_k \ge \ln(1+k) / \ln(1/\beta)$. With this choice, we have $c = \frac{1}{2}$ in~\eqref{eq:absorb-provided} and $e^{2c} < 3$. Setting
$
S_k = \frac{\gamma}{4}\|\bar{g}^{(k)}\|^2 + \frac{\gamma}{6}\|\nabla f(\bar{x}^{(k)})\|^2 + \frac{12}{\gamma}\|\Delta_x^{(k+1)}\|_F^2,  D_k = f(\bar{x}^{(k)}) - f^*, F_k = \frac{6}{\gamma}\|\Delta_x^{(k)}\|_F^2,
$
and applying~\eqref{eq:absorbing-4}, we obtain
\begin{equation}
\begin{aligned}
&\quad \frac{\gamma}{4}\sum_{k=0}^{K-1}\|\bar{g}^{(k)}\|^2+\frac{\gamma}{6}\sum_{k=0}^{K-1}\|\nabla f(\bar{x}^{(k)})\|^2 +\frac{12}{\gamma}\sum_{k=1}^{K}\|\Delta_x^{(k)}\|_F^2\\
&\le 3(f(\bar{x}^{(0)})-f^*)+\frac{12}{\gamma}\sum_{k=0}^{K-1}\|\Delta_x^{(k)}\|_F^2.
\end{aligned}
\label{eq:ready-to-prove-thm}
\end{equation}
Note that Algorithm~\ref{alg:PWM-GD} is initialized at consensus, so $\Delta_x^{(0)} = 0$. Substituting the definitions of $C_A$, $C_W$, and $\beta$ into the above bounds, we obtain our main result:
\begin{theorem}[\sc Convergence Property]\label{thm:main}
    Suppose Assumptions~\ref{ass:tv graph}--\ref{ass:smooth} hold for Algorithm~\ref{alg:PWM-GD}. If we choose step size 
    $
    \gamma \le \frac{(1-\eta)^2}{16n^2(1+12n)L},
    $
    and the inner iteration rounds
    \[
    \begin{aligned}
    R_k & \ge\max\{\ln(3C_A)/\ln(1/\beta),\ln (C_W)/\ln(1/\beta),\ln(1+k)/\ln(1/\beta)\}\\
    &= \max\Big\{B+\frac{B}{-\ln(1-\eta)}\big(\frac{3}{2}\ln n +\ln(3+\frac{3}{\eta})\big), \frac{B\ln(1+k)}{-\ln(1-\eta)} \Big\},
    \end{aligned}
    \]
    where $\eta$ and $B$ are defined in Proposition~\ref{prop:positive}, we have
    \begin{align}
    \frac{1}{K}\sum_{k=0}^{K-1}\norm{\nabla f(\bar{x}^{(k)})}^2\le \frac{18(f(\bar{x}^{(0)})-f^*)}{\gamma K},
    \label{ieq:main}
    \end{align}
    where \(f^*:=n^{-1}\sum_{i=1}^n\inf_{x\in\mathbb{R}^d}f_i(x)\).
\end{theorem}

With the smallest admissible $R_k$, the total number of communication rounds is
\[\sum_{k=0}^{K-1}R_k=\sum_{k=0}^{K-1}\Big\lceil\frac{\max\{\ln(k+1),\ln(3C_A)\}}{\ln(1/\beta)}\Big\rceil\le\frac{K\max\{\ln K,\ln(3C_A)\}}{\ln(1/\beta)}+K. \]

\begin{proof}
{
Based on Lemma~\ref{lemma:descent}, \ref{lemma:consensus-iteration} and \ref{lemma:absorbing-2}, we obtain the inequality \eqref{eq:ready-to-prove-thm}. As Algorithm~\ref{alg:PWM-GD} is initialized at consensus, namely $\Delta_x^{(0)}=0$, the terms $\sum_{k=1}^{K-1}\|\Delta_x^{(k)}\|_F^2$ and \(\sum_{k=0}^{K-1}\|\Delta_x^{(k)}\|_F^2 \) on both sides are canceled. The non-negative terms $\sum_{k=0}^{K-1}\|\bar{g}^{(k)}\|^2$ and $\|\Delta_x^{(K)}\|_F^2$ can be directly omitted. And then the main inequality \eqref{ieq:main} follows. The choice of step size $\gamma$ and the inner iteration rounds $R_k$ are determined by Lemma~\ref{lemma:descent}, \ref{lemma:consensus-iteration}, and the sufficient choice when utilizing Lemma~\ref{lemma:absorbing-2}.}
\end{proof}

\begin{remark}[\sc Convergence in Communication Rounds] 
In Theorem~\ref{thm:main}, we define the total number of communication rounds as $T=\sum_{k=0}^{K-1}R_k$ and obtain $K = \Omega(T/\ln(T))$. Substituting this into~\eqref{ieq:main} yields: 
\begin{align}
    \frac{1}{K}\sum_{k=0}^{K-1}\norm{\nabla f(\bar{x}^{(k)})}^2 = \mathcal{O}\left(\frac{\ln(T)}{T}\right).
\end{align}
This complexity is nearly the same order as decentralized optimization over doubly-stochastic~\cite{alghunaim2022unified} or column-stochastic mixing matrices~\cite{liang2023towards}, except for an additional $\ln(T)$ factor. To the best of our knowledge, PULM-DGD is the first approach that achieves convergence for decentralized optimization over TVBNs with highly dynamic row-stochastic mixing matrices. Equivalently, $\frac{1}{K}\sum_{k=0}^{K-1}\|\nabla f(\bar{x}^{(k)})\|^2\le\epsilon$ is attained within $T=\mathcal{O}(\epsilon^{-1}\ln(1/\epsilon))$ communication rounds.
\end{remark}

{ 
\begin{remark}[\sc A constant $R_k$ does not suffice for exact convergence]\label{rmk:counterexample}
The following counterexample shows that a constant $R_k$ does not suffice
for PULM-DGD to converge exactly, even over a static network. Let $n=2$ and
$d=1$, and consider the time-invariant mixing matrix and local objectives 
\begin{equation}\label{eq:ce-A}
A^{(t)}\equiv A:=\begin{pmatrix}\frac45&\frac15\\[2pt]\frac25&\frac35\end{pmatrix}, \qquad f_1(x)=\tfrac12x^2,\qquad f_2(x)=\tfrac32(x-1)^2,
\end{equation}
so that $f=\frac12f_1+\frac12f_2$ has the unique stationary point $x^\star=\frac34$.
Assumptions~\ref{ass:tv graph}--\ref{ass:smooth} hold with $\tilde B=1$,
$\tau=\frac15$ and $L=3$, and Proposition~\ref{prop:positive} holds with $B=1$ and
$\eta=\frac15$. Let $R_k\equiv R$. For $R=1$ and every $\gamma\in(0,1]$, the
iterates converge from any initialization to
\begin{equation}\label{eq:ce-xinf}
\vx^{\infty}=\frac{1}{51\gamma+82}\begin{pmatrix}54\\[1pt]51\gamma+54\end{pmatrix},
\qquad
\nabla f(\bar x^{\infty})=-\frac{51\gamma+30}{2(51\gamma+82)},
\end{equation}
so that $|\nabla f(\bar x^{\infty})|\ge\frac{15}{82}$ for every such $\gamma$. For
every constant $R\ge1$, there exists $\gamma_R>0$ such that, for every
$\gamma\in(0,\gamma_R)$, the iterates converge to a limit $\vx^{\infty}(\gamma)$,
and $\vx^{\infty}(\gamma)\to x_0(R)\one$ as $\gamma\to0^+$, where
\begin{equation}\label{eq:ce-gap}
\nabla f\big(x_0(R)\big)=\frac{v/10-3u/5}{4-6u/5-v/15},
\qquad
\tfrac18\big(\tfrac45\big)^{R-1}\ \le\ \big|\nabla f\big(x_0(R)\big)\big|\ \le\
\big(\tfrac45\big)^{R-1},
\end{equation}
with $u:=(\frac45)^{R-1}$ and $v:=(\frac35)^{R-1}$. The gap therefore decays
geometrically in $R$ but is nonzero for every constant $R$, and by continuity the
limit is not stationary for all sufficiently small $\gamma$. The gap vanishes only
as $R\to\infty$: no constant $R_k$ suffices, and the growing schedule of
Theorem~\ref{thm:main} is sufficient. The proof is given in
Appendix~\ref{app:counterexample}.
\end{remark}
}

\section{Experiments}\label{sec:exp}

We evaluate PULM (Algorithm~\ref{alg:PWM}) and PULM-DGD (Algorithm~\ref{alg:PWM-GD}) on
three problem classes: average consensus, logistic regression with a nonconvex
regularizer, and training convolutional neural networks on MNIST and CIFAR-10. The
experiments address three questions: whether PULM-type methods remain convergent over time-varying row-stochastic digraphs, where push-type methods lose column-stochasticity
(Section~\ref{sec:push-comp}); how the number of inner rounds $R_k$ affects PULM-DGD in
practice (Section~\ref{sec:diff-inner}); and how the network topology and the data
distribution affect both methods (Section~\ref{sec:topo-exp}).

\subsection{Experimental Setup}\label{sec:exp-setup}

\textbf{Time-varying topologies.}
We generate two families of time-varying digraphs. In \emph{random topologies}, at each
round $k$ every node $i$ transmits to every other node $j$ independently with probability
$p_c$. In \emph{latent topologies}, we fix a strongly connected digraph, either a directed
ring or a random digraph of sparsity $s$, and at each round $k$ every edge of this graph is
removed independently with probability $p_d$. In both families the sender knows which
links are active before transmission, so the resulting out-degrees are available to
push-type methods.

\vspace{1mm}
\noindent\textbf{Packet loss.}
To model the TVBN setting, the experiments in Section~\ref{sec:push-comp} additionally
drop each transmitted message independently with probability $p_t$ after transmission
and before reception. Unlike the disconnections governed by $p_c$ and $p_d$, such losses
cannot be anticipated by the sender and are not reflected in its out-degree. A node that
scales its outgoing weights by its nominal out-degree therefore produces a mixing matrix
that is no longer column-stochastic, whereas the row-stochastic weights, which each node
forms from the messages it actually receives, remain valid.

\vspace{1mm}
\noindent\textbf{Average consensus.}
Each node $i$ holds a vector $x_i\in\mathbb{R}^d$, and the goal is to compute
$n^{-1}\sum_{i=1}^n x_i$. We measure the relative error
\begin{equation}\label{eq:x-error}
e^{(k)}:=\frac{\|\vz^{(k)}-\bar{\vx}\|_F}{\|\vx-\bar{\vx}\|_F},
\qquad \bar{\vx}:=\fulcon\vx ,
\end{equation}
where $\vz^{(k)}$ and $\vx$ are as in Algorithm~\ref{alg:PWM}, and also report
$\|W^{(k)}-\fulcon\|_F$ where relevant.

\vspace{1mm}
\noindent\textbf{Logistic regression.}
Given samples $\{(x_s,y_s)\}_{s=1}^S$ with $x_s\in\mathbb{R}^d$ and $y_s\in\{+1,-1\}$,
partitioned into local datasets $\{(x_{i,s},y_{i,s})\}_{s=1}^{S_i}$, node $i$ holds
\[
f_i(w,b):=\frac{1}{S_i}\sum_{s=1}^{S_i}\log\Big(1+\exp\big(-y_{i,s}(x_{i,s}^\top w+b)\big)\Big)
+\lambda\bigg(\frac{b^2}{1+b^2}+\sum_{j=1}^d\frac{w_j^2}{1+w_j^2}\bigg),
\]
where the second term is a nonconvex regularizer and $\lambda=0.1$ throughout. Synthetic
data are generated as follows: we draw global parameters $(w^*,b^*)$, then local
parameters $w_i^*\sim\mathcal{N}(w^*,\sigma_h^2 I_d)$ and
$b_i^*\sim\mathcal{N}(b^*,\sigma_h^2)$ with heterogeneity level $\sigma_h\ge0$; for each
local sample we draw $x_{i,s}$ at random and set $y_{i,s}=+1$ with probability
$\sigma(x_{i,s}^\top w_i^*+b_i^*)$ and $y_{i,s}=-1$ otherwise, where
$\sigma(t)=1/(1+e^{-t})$. We use $n=10$, $d=30$, $\sigma_h=0.1$, and $1000$ samples per
node. For real data, we use the UCI Human Activity Recognition
dataset\footnote{\url{https://archive.ics.uci.edu/dataset/240/human+activity+recognition+using+smartphones}},
split into $10$ subsets, with class $1$ as the positive label.

\vspace{1mm}
\noindent\textbf{Neural network training.}
We train LeNet on MNIST and ResNet-18 on CIFAR-10 over $10$ nodes with batch size $100$.
The training set is partitioned either uniformly at random or by label, in which case
node $i$ receives the samples of class $i-1$. These experiments use minibatch stochastic
gradients and thus go beyond the deterministic setting of Theorem~\ref{thm:main}.

\vspace{1mm}
\noindent\textbf{Consensus error in optimization.}
For PULM-DGD we additionally report the consensus error
\[
\epsilon^{(k)}:=\sqrt{\frac{1}{d}\sum_{j=1}^d \Big(\frac{1}{n}\sum_{i=1}^n
\big(x^{(k)}_{i,j}-\bar{x}^{(k)}_{j}\big)^2\Big)^2},
\qquad \bar{x}_j^{(k)}:=\frac{1}{n}\sum_{i=1}^n x_{i,j}^{(k)} .
\]

\subsection{Comparison with Push-Type Methods}\label{sec:push-comp}

\pushdg~\cite{nedic2017achieving} and its consensus counterpart
\pushsum~\cite{kempe2003gossip} are among the few methods designed for time-varying
directed networks. Both rely on out-degree information to form column-stochastic weights
and are therefore exposed to the packet loss described in Section~\ref{sec:exp-setup},
whereas PULM and PULM-DGD do not use out-degrees.

\begin{figure}[t]
    \centering
    \includegraphics[width=1\linewidth]{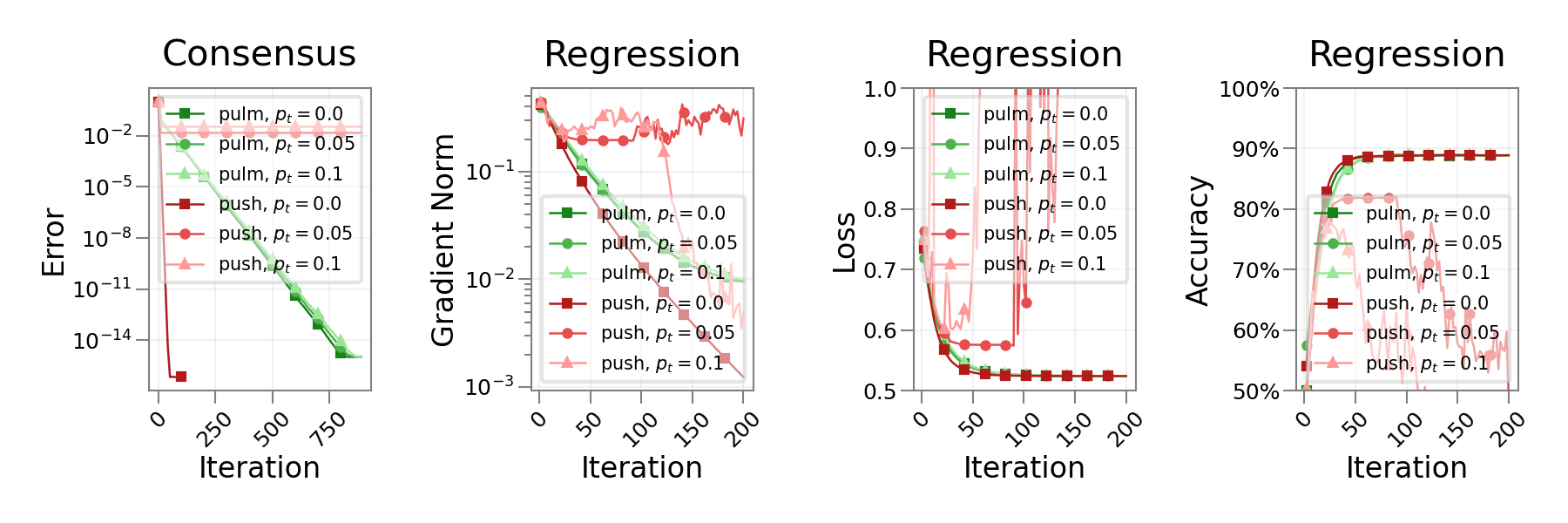}
    \caption{\small Push-type versus PULM-type methods under packet loss on time-varying
    directed networks. In consensus, \texttt{push} denotes \pushsum and \texttt{pulm}
    denotes PULM; in regression, \texttt{push} denotes \pushdg and \texttt{pulm} denotes
    PULM-DGD. The packet-loss probability is $p_t$.}
    \label{fig:compare_pwm_and_dig}
\end{figure}

Figure~\ref{fig:compare_pwm_and_dig} reports the results. For consensus, we use $n=20$
nodes and random data of dimension $d=1024$ over a random topology with $p_c=0.2$. For
regression on synthetic data, we use a latent topology with $s=0.3$ and $p_d=0.4$, step
size $\gamma=0.1$, and $R_k\equiv10$. The disconnections governed by $p_d$ are reflected
in the out-degrees and serve only to make the topology time-varying; the packet-loss
probability $p_t\in\{0,0.05,0.1\}$ is the parameter of interest. For $p_t=0$, where the
out-degrees are exact, both families converge. When $p_t > 0$, the effective out-degree changes unexpectedly, and column-stochasticity can no longer be guaranteed. When $p_t=0.05$, \pushsum and \pushdg
already fail to converge, whereas PULM and PULM-DGD converge for every tested $p_t$.
Figure~\ref{fig:compare-dig-real} repeats the regression comparison on real data with
$\gamma=0.01$ and all other settings unchanged, and leads to the same conclusion.

\begin{figure}[t]
    \centering
    \includegraphics[width=0.85\linewidth]{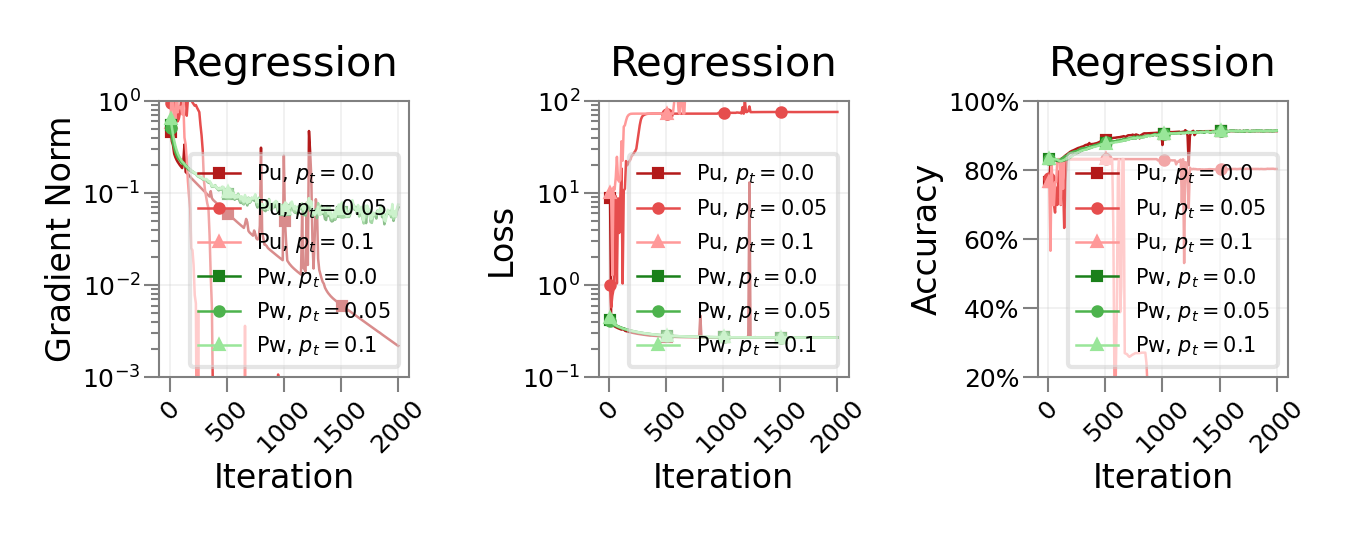}
    \caption{\small \pushdg versus PULM-DGD under packet loss for logistic regression on
    real data.}
    \label{fig:compare-dig-real}
\end{figure}

\subsection{Effect of the Inner Communication Rounds \texorpdfstring{$R_k$}{Rk}}\label{sec:diff-inner}

The number of inner rounds $R_k$ determines the amount of mixing per outer iteration and
hence the communication cost. Theorem~\ref{thm:main} requires $R_k$ to grow as $\ln(k)$;
we examine how PULM-DGD behaves with constant $R_k$ in practice. We use a latent topology
with $s=0.3$ and $p_d=0.4$, a random partition of the data over $10$ nodes, and step size
$0.01$ for logistic regression on real data and $0.1$ for all other problems.

\begin{figure}[t]
    \centering
    \includegraphics[width=1\linewidth]{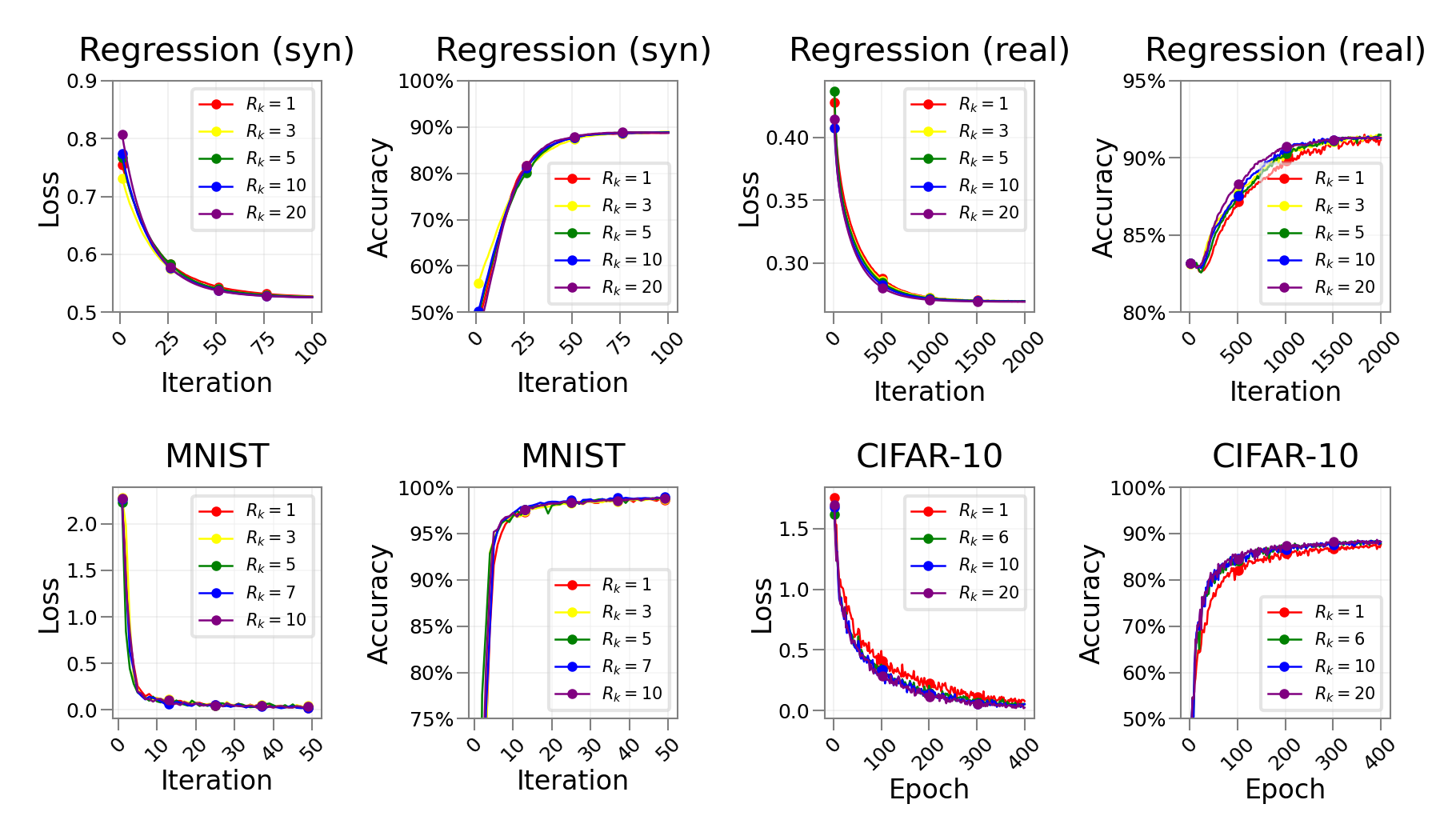}
    \caption{\small Effect of $R_k$ on PULM-DGD. Top row: logistic regression on synthetic
    data (loss and accuracy) and on real data (loss and accuracy) for
    $R_k\in\{1,3,5,10,20\}$. Bottom row: MNIST training (loss and accuracy) for
    $R_k\in\{1,3,5,7,10\}$, and CIFAR-10 training (loss and accuracy) for
    $R_k\in\{1,6,10,20\}$.}
    \label{fig:diff-inner}
\end{figure}

As shown in Figure~\ref{fig:diff-inner}, the loss and accuracy curves for different
$R_k$ nearly coincide on all four problems. Larger $R_k$ slightly accelerates the early
iterations, but this gain saturates quickly, and $R_k=1$ already yields comparable final
performance. Figures~\ref{fig:inner-logi-syn}--\ref{fig:inner-cifar} report the results
for each problem in detail, including the consensus error $\epsilon^{(k)}$. The
consensus error decreases as $R_k$ increases, but with diminishing returns: each
additional inner round yields a smaller reduction than the previous one. The residual
consensus error left by a small $R_k$ therefore does not limit the optimization
performance on these problems.

This observation does not contradict the counterexample in
Remark~\ref{rmk:counterexample}. There, PULM-DGD with any constant $R_k$ converges,
for small step sizes, to a limit whose average is not stationary: no constant $R_k$
suffices, and the growing schedule of Theorem~\ref{thm:main} is sufficient. The resulting
gap, however, decays geometrically in $R_k$; as $\gamma\to0^+$ it is at most
$(\frac45)^{R_k-1}$. On the
networks considered here, the residual bias induced by a small constant $R_k$ is
too small to affect the loss and
accuracy.

\begin{figure}[t]
    \centering
    \includegraphics[width=1\linewidth]{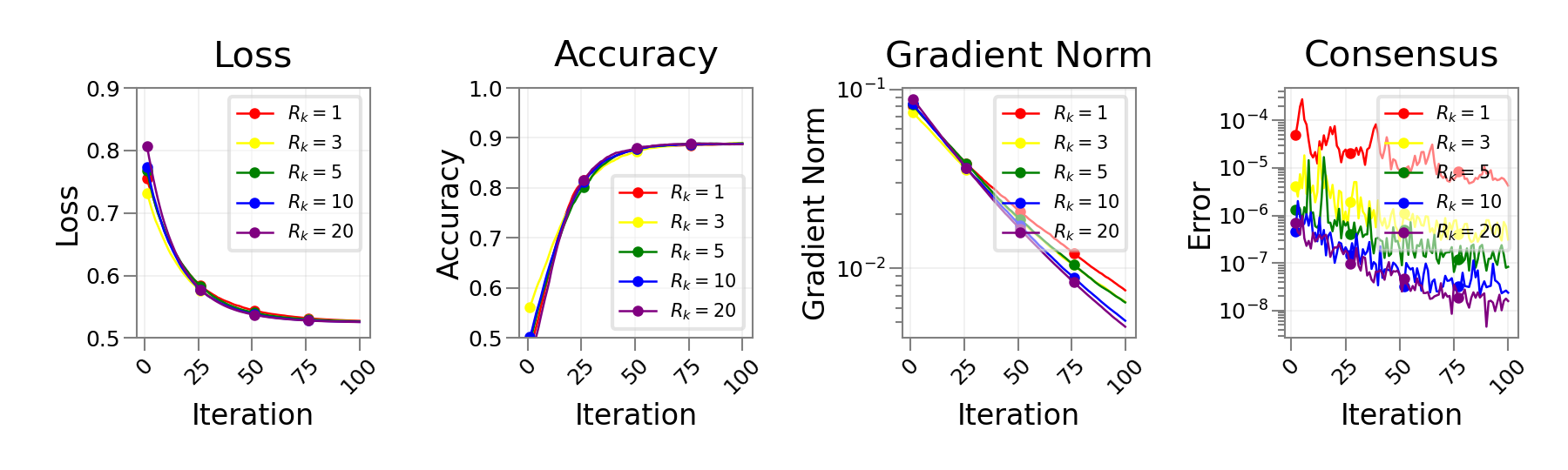}
    \caption{\small Effect of $R_k$ on PULM-DGD for logistic regression on synthetic data.}
    \label{fig:inner-logi-syn}
\end{figure}

\begin{figure}[t]
    \centering
    \includegraphics[width=1\linewidth]{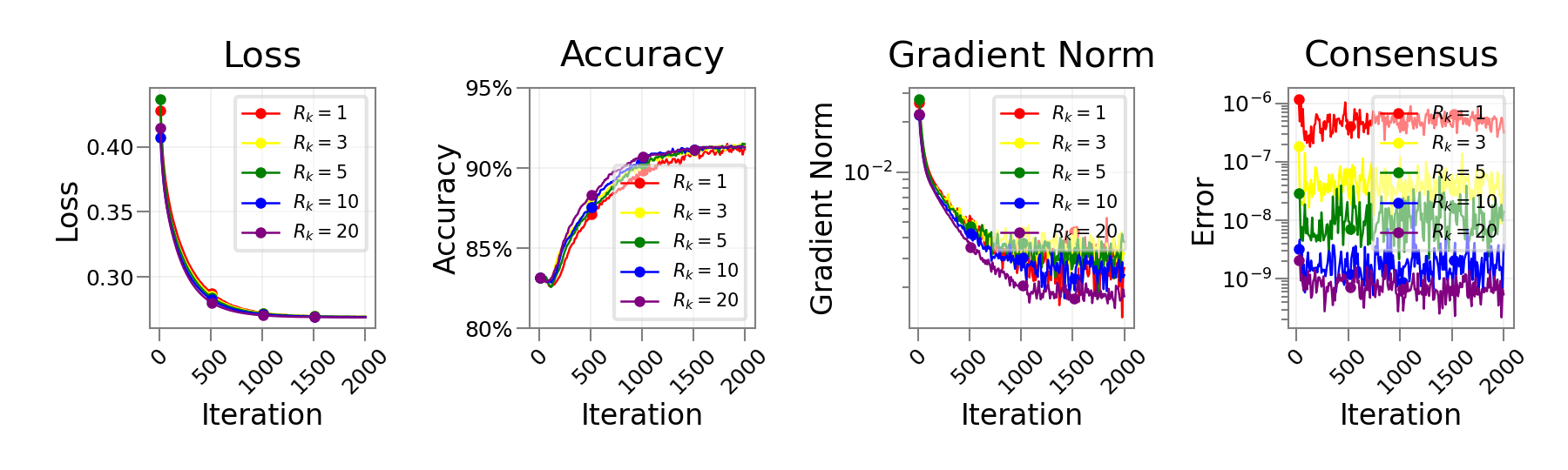}
    \caption{\small Effect of $R_k$ on PULM-DGD for logistic regression on real data.}
    \label{fig:inner-logi-real}
\end{figure}

\begin{figure}[t]
    \centering
    \includegraphics[width=1\linewidth]{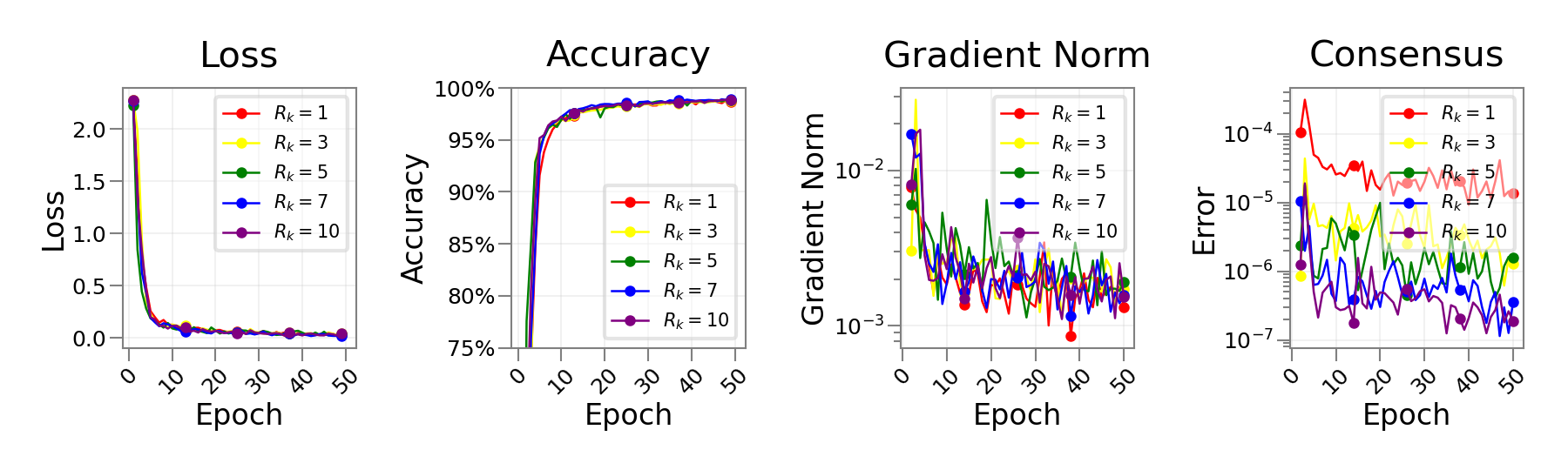}
    \caption{\small Effect of $R_k$ on PULM-DGD for MNIST training.}
    \label{fig:inner-mnist}
\end{figure}

\begin{figure}[t]
    \centering
    \includegraphics[width=1\linewidth]{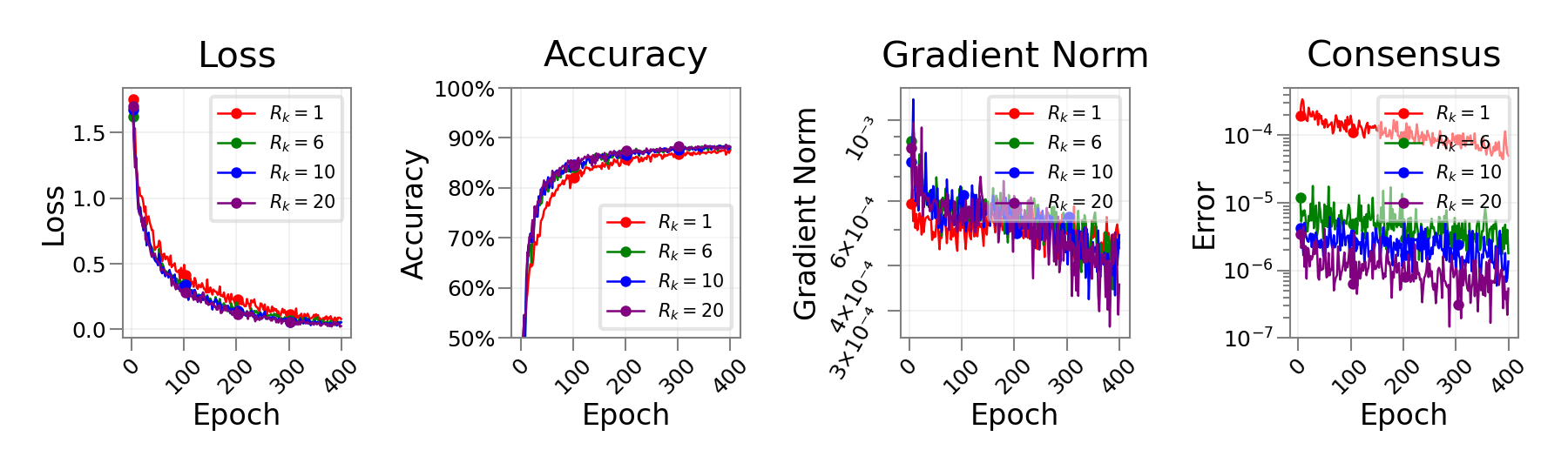}
    \caption{\small Effect of $R_k$ on PULM-DGD for CIFAR-10 training.}
    \label{fig:inner-cifar}
\end{figure}

\subsection{Performance on Different Topologies}\label{sec:topo-exp}

We finally examine how the topology affects both methods. We consider seven topologies
over $n=20$ nodes: (1) a latent ring with $p_d=0.2$; (2) a latent graph with $s=0.2$ and
$p_d=0.2$; (3) a latent graph with $s=0.3$ and $p_d=0.2$; (4) a latent graph with
$s=0.3$ and $p_d=0.4$; and (5)--(7) random topologies with $p_c=0.1$, $0.2$, and $0.3$.

\begin{figure}[t]
    \centering
    \includegraphics[width=1\linewidth]{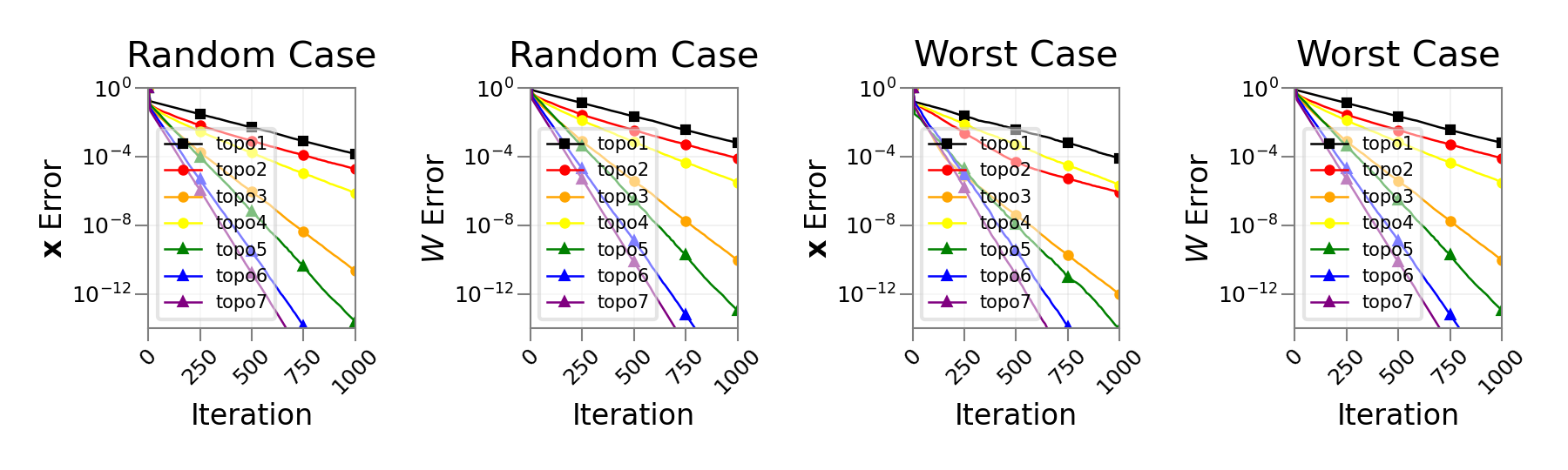}
    \caption{\small PULM for average consensus on seven time-varying topologies, with
    random and worst-case data.}
    \label{fig:topo-consensus}
\end{figure}

For consensus, we consider two data configurations: all $n$ vectors drawn at random, and
$n-1$ vectors drawn at random with the remaining one placed far from the others, which
serves as a worst case. Figure~\ref{fig:topo-consensus} reports both $e^{(k)}$ and
$\|W^{(k)}-\fulcon\|_F$. PULM behaves similarly in the two configurations, consistent
with Theorem~\ref{thm:exp convergence}, whose rate does not depend on the data. Random
topologies converge faster than latent ones, which is consistent with their having a
smaller $B$ or a larger $\eta$ in Proposition~\ref{prop:positive}.

\begin{figure}[t]
    \centering
    \includegraphics[width=1\linewidth]{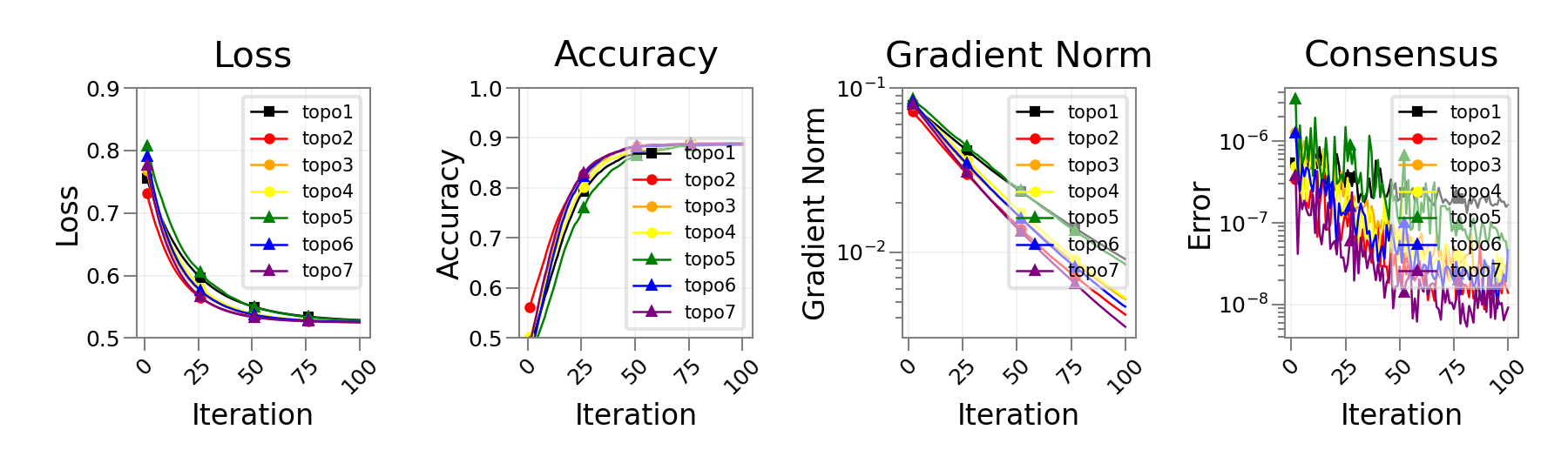}
    \caption{\small PULM-DGD for logistic regression on synthetic data over seven
    topologies.}
    \label{fig:topo-logi-syn}
\end{figure}

\begin{figure}[t]
    \centering
    \includegraphics[width=1\linewidth]{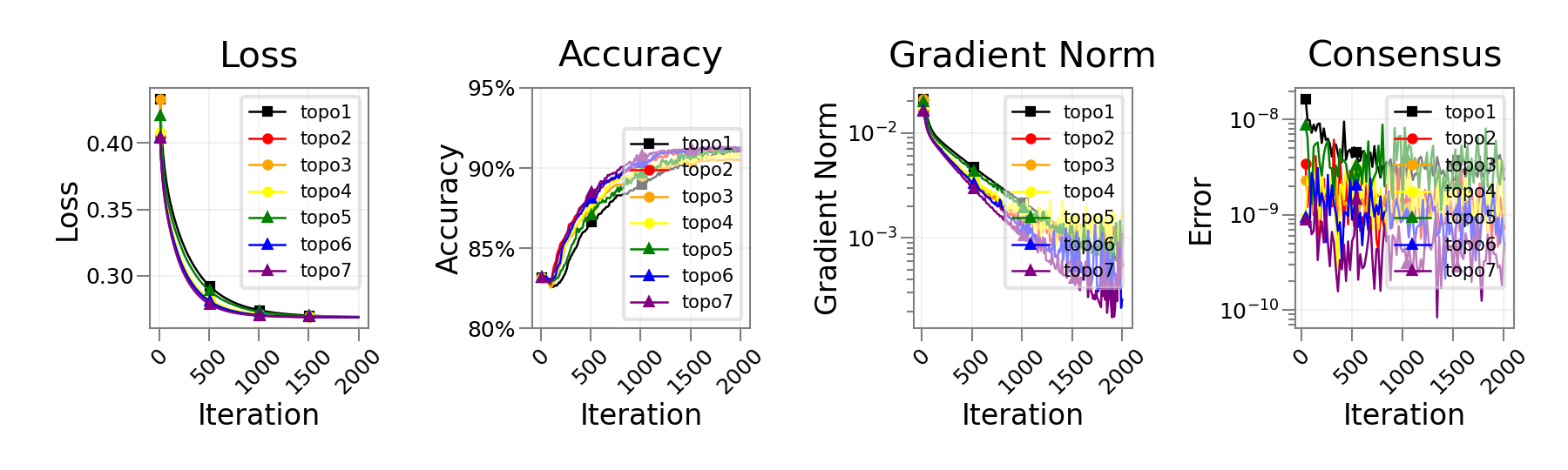}
    \caption{\small PULM-DGD for logistic regression on real data over seven topologies.}
    \label{fig:topo-logi-real}
\end{figure}

Figures~\ref{fig:topo-logi-syn} and~\ref{fig:topo-logi-real} report PULM-DGD on the same
seven topologies for logistic regression, with $R_k\equiv10$ and step size $0.1$ on
synthetic data and $0.01$ on real data. Topologies that mix faster yield a smaller
consensus error $\epsilon^{(k)}$, but the loss and accuracy differ little across
topologies, indicating that moderately accurate consensus is sufficient for these
nonconvex problems.

\begin{figure}[t]
    \centering
    \includegraphics[width=1\linewidth]{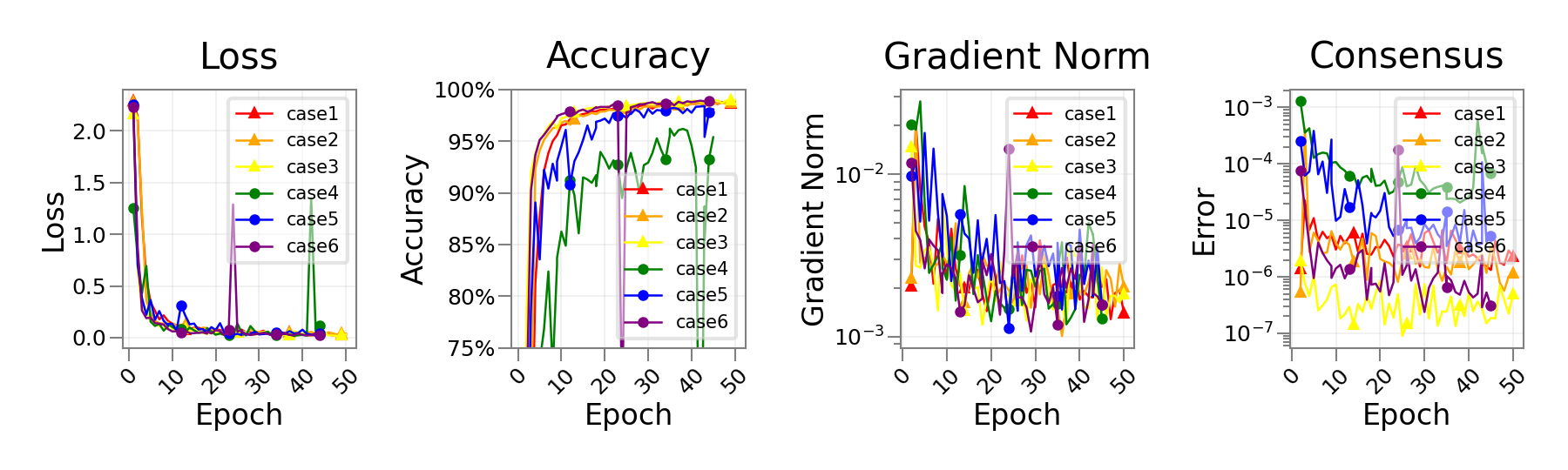}
    \caption{\small PULM-DGD for MNIST training under three topologies and two data
    partitions.}
    \label{fig:topo-mnist}
\end{figure}

For MNIST (Figure~\ref{fig:topo-mnist}), we combine three topologies, namely a latent
ring with $p_d=0.2$, a latent graph with $s=0.3$ and $p_d=0.4$, and a random topology
with $p_c=0.3$, with the two data partitions of Section~\ref{sec:exp-setup}, using step
size $0.1$. Cases~1--3 use the random partition over the latent ring, the latent graph,
and the random topology, respectively, and Cases~4--6 use the label-based partition over
the same three topologies. Under the label-based partition, nodes hold different numbers
of samples and complete epochs at different times; we count epochs by the slowest node.
PULM-DGD converges in all six cases, and the topology and the data partition affect
convergence more than $R_k$ does (cf.\ Section~\ref{sec:diff-inner}).

\section{Conclusion}\label{sec:conclusion}
We studied decentralized optimization over time-varying broadcast networks (TVBNs), where only time-varying row-stochastic mixing matrices are available. Since the limit vector of the mixing process depends on future graph realizations, bias-correction strategies that estimate it are ruled out. We therefore proposed PULM, which controls the limit instead of estimating it: a communication-free adjustment step anchors the weight of each node's own initial vector at $1/n$, yielding average consensus at an exponential rate for every admissible graph sequence. Building on PULM, PULM-DGD finds an $\epsilon$-stationary solution of smooth nonconvex problems within $\cO(\epsilon^{-1}\ln(1/\epsilon))$ communication rounds, matching methods over doubly-stochastic or column-stochastic networks up to a logarithmic factor. PULM-DGD requires a growing number of inner rounds, which Remark~\ref{rmk:counterexample} shows cannot be avoided in general, and its constants can be large on poorly connected graph sequences. Developing single-loop algorithms over time-varying row-stochastic digraphs, for instance by combining the anchoring principle with gradient tracking, is a natural direction for future work.


\appendix
\setcounter{equation}{0}
\renewcommand{\theequation}{A\arabic{equation}}
\renewcommand{\theHequation}{App.equation.\arabic{equation}}

\section{Additional Proofs}

\subsection{Proof of Proposition~\ref{prop:positive}}\label{sec:proof-prop-positive}
\begin{proof}
We introduce the following notation: (i) $v_1 \rightarrow v_2$ indicates that there exists an edge from $v_1$ to $v_2$ in a given graph $\mathcal{G}$ (where $\mathcal{G}$ may be an accumulated graph as defined in Assumption~\ref{ass:tv graph}); (ii) $v_1 \xrightarrow{t} v_2$ indicates that $v_1$ sends information to $v_2$ at time $t$, i.e., $(v_1, v_2) \in \mathcal{E}^{(t)}$.

It suffices to show that for any starting time $k$, any source node $s \in \mathcal{V}$, and any destination node $e \in \mathcal{V}$, there exists a finite-length path from $s$ to $e$ whose length is independent of $k$. By Assumption~\ref{ass:tv graph}, self-loops exist at every node. Let $B$ denote the maximum path length among all source-destination pairs, and extend each shorter path to length $B$ by appending self-loops. Applying the Chapman--Kolmogorov equation, we obtain

\[
\bigg[\prod_{l=k}^{k+B-1}A^{(l)}\bigg]_{e,s}\ge \big[A^{(k+B-1)}\big]_{e,v_{B-1}}\cdots\big[A^{(k+1)}\big]_{v_2,v_{1}}\big[A^{(k)}\big]_{v_1,s}\ge \tau^B,
\]
where $s\xrightarrow{k}v_1\xrightarrow{k+1}v_2\cdots\xrightarrow{k+B-2}v_{B-1}\xrightarrow{k+B-1}e$ is a given state transition trajectory with length $B$, and the second inequality is from Assumption~\ref{ass:matrix}.

Since directed graphs always contain self-loops, the case $s = e$ is trivial. For $s \neq e$, we define $\mathcal{V}_{\ominus}$ as the set of nodes that have already been visited, initialized as $\mathcal{V}_{\ominus} = \{s\}$, and expand this set during the search process.

The searching process can be broken down into the following steps.

\textbf{Step 1.} Use Assumption~\ref{ass:tv graph}, during $\tilde{B}$ time intervals, we can find at least one path from $s$ to $e$, from which we pick out a shortest path $s\to v_1\to \dots\to e$. We add $v_1$ into $\mathcal{V}_{\ominus}$, and there exists at least one time point $t_1\in[k,k+\tilde{B}-1]$ such that $(s,v_1)\in\mathcal{E}^{(t_1)}$. We construct the trajectory as 
\[
s\xrightarrow{k}s\xrightarrow{k+1}\cdots s\xrightarrow{t_1-1}s\xrightarrow{t_1}v_1,
\]
and record the time point $t_1$. Then go to Step 2.

\textbf{Step 2.} Rename the latest added node as $v_-$ and the latest recorded time point as $t_-$. If $v_-=e$, go to Step 4. Otherwise, go to Step 3.

\textbf{Step 3.} Changing starting node as $v_-$ and time point as $t_-+1$, using Assumption~\ref{ass:tv graph} again, during $\tilde{B}$ time intervals, we can find at least one path from $v_-$ to $e$, from which we pick out a shortest path $v_-\to \dots \to v_\times \to v_+\to \dots\to e$, where $v_\times$ is the last node which still stays in $\mathcal{V}_{\ominus}$ and $v_+$ is the first node not in $\mathcal{V}_{\ominus}$. We denote the first time point when $v_\times$ is searched as $t_\times$. Then there exists at least one time point $t_+\in[t_- +1, t_- +\tilde{B}]$ such that $(v_\times, v_+)\in\mathcal{E}^{(t_+)}$. We continue the trajectory from time point $t_\times$ as
\[
\cdots \xrightarrow{t_\times}v_\times \xrightarrow{t_\times +1}\cdots v_\times\xrightarrow{t_+-1}v_\times  \xrightarrow{t_+}v_+.
\]
It should be noticed that if $v_\times \neq v_-$, we have $t_\times<t_-$. In this case, the state transition trajectory is partially reconstructed. After adding $v_+$ into $\mathcal{V}_{\ominus}$ and recording the new time point $t_+$, we go back to Step 2.

\textbf{Step 4.} The searching process is finished, and a state transition trajectory has been constructed.

As the set of searched nodes $\mathcal{V}_{\ominus}$ expands continuously, Step 3 will be repeated no more than $n-1$ times. We can find a state transition trajectory starting at time point $k$ from $s$ to $e$ whose length $B$ is no longer than $n\tilde{B}$. Together with the claim at the beginning of the proof, the existence of $B$ and $\eta$ is proved.
\end{proof}

\subsection{Linear Convergence Lemma for Positive Mixing Matrices}\label{sec(app):A-limit}

\begin{lemma}
    \label{lem(app):A-limit}
    For a sequence of row-stochastic matrices $\tilde{A}^{(1)}, \tilde{A}^{(2)}, \dots\in\RR^{n\times n}$, assuming $\tilde{A}^{(t)}_{i,j}\geq\eta>0,~\forall i,j\in[n],~\forall t=1,2,\dots$, we have
\begin{enumerate}[label=(\roman*)]
    \item The limit $\tilde{A}^{\infty}:=\prod_{t=1}^{+\infty} \tilde{A}^{(t)}$ exists. And $\tilde{A}^{\infty}$ is a row-stochastic matrix with identical rows, i.e.,
    \[
        \tilde{A}^{\infty}=\one \pi^\top,
    \]
    where the limit vector $\pi$ is determined by the whole sequence \(\{\tilde{A}^{(t)}\}_{t=1}^\infty\).
    \item The maximum deviation converges with geometric rate, i.e.,
    \[
        \max\bigg|\prod_{t=1}^T \tilde{A}^{(t)}-\tilde{A}^{\infty}\bigg|\leq \frac{1+\eta}{\eta}(1-\eta)^T
    \]
\end{enumerate}
\end{lemma}

\begin{proof}
This lemma and its proof are Lemma~3 of~\cite{nedic2009distributed} (see also Lemma~4 of~\cite{nedic_convergence_2010}), whose qualitative part goes back to Tsitsiklis~\cite{tsitsiklis1984problems}; we include the argument to keep the constants explicit.

\textbf{Proof of (i).} Consider an arbitrary vector $\boldsymbol{x}^{(0)}\in\RR^{n}$, we suppose to prove that the limit of $\boldsymbol{x}^{(t)}:=\tilde{A}^{(t)}\cdots \tilde{A}^{(1)}\boldsymbol{x}^{(0)}$ exists. Note that $\boldsymbol{x}^{(t)}$ can be decomposed into a consensus part and a surplus part, i.e.,
\begin{equation}
    \boldsymbol{x}^{(t)}=p^{(t)}\one +\boldsymbol{z}^{(t)},\quad \boldsymbol{z}^{(t)}\geq 0.\label{two-part-x}
\end{equation}
where $p^{(t)}$ is a scalar and  $\boldsymbol{z}^{(t)}$ is a vector. Our goal is to construct a Cauchy sequence $\{p^{(t)}\}_{t=0}^\infty$ and $\{\bm{z}^{(t)}\}_{t=0}^\infty\to \bm{0}$. We first choose
\begin{equation*}
    p^{(0)}=\min_{1\leq i \leq n}\{\boldsymbol{x}^{(0)}_i\},\quad \boldsymbol{z}^{(0)}=\boldsymbol{x}^{(0)}-\min_{1\leq i \leq n}\{\boldsymbol{x}^{(0)}_i\}\one.
\end{equation*}
Using the decomposition, the recursion of $\boldsymbol{x}^{(t)}$ can be expanded as
\begin{align*}
    \boldsymbol{x}^{(t+1)}=\tilde{A}^{(t+1)}\boldsymbol{x}^{(t)}=\tilde{A}^{(t+1)}(p^{(t)}\one +\boldsymbol{z}^{(t)})=p^{(t)}\one +\tilde{A}^{(t+1)}\boldsymbol{z}^{(t)}.
\end{align*}
We define the index $i^*$ as 
\begin{align*}
    i^*:=\mathop{\arg\min}_{1\leq i \leq n} \tilde{A}^{(t+1)}_{i,\cdot} \boldsymbol{z}^{(t)},
\end{align*}
which means $\min_{1\leq i\leq n} \{\tilde{A}^{(t+1)}_{i,\cdot} \boldsymbol{z}^{(t)}\}=\tilde{A}^{(t+1)}_{i^*,\cdot} \boldsymbol{z}^{(t)}$. And we further specify $\boldsymbol{z}^{(t+1)}$  and $p^{(t+1)}$ as
\begin{align*}
    \boldsymbol{z}^{(t+1)}=\tilde{A}^{(t+1)}\boldsymbol{z}^{(t)}-\big(\tilde{A}^{(t+1)}_{i^*,\cdot} \boldsymbol{z}^{(t)}\big)\one,\quad p^{(t+1)}=p^{(t)}+\tilde{A}^{(t+1)}_{i^*,\cdot} \boldsymbol{z}^{(t)}.
\end{align*}
Because matrices $A^{(t)}$ are row-stochastic, according to our choice of $p^{(t)}$ and $\boldsymbol{z}^{(t)}$, all $\boldsymbol{z}^{(t)}$ are non-negative with minimum zero.

Next we prove the geometric convergence of $\boldsymbol{z}^{(t)}$. Consider each entry of $\boldsymbol{z}^{(t+1)}$, for $i=1,\dots,n$,
\begin{align*}
    \boldsymbol{z}^{(t+1)}_i=\tilde{A}^{(t+1)}_{i,\cdot} \boldsymbol{z}^{(t)}-\tilde{A}^{(t+1)}_{i^*,\cdot} \boldsymbol{z}^{(t)}.
\end{align*}
We divide the indices $j=1,\dots, n$ into two subsets of $\{1,\dots,n\}$:
\begin{align*}
    J^+ := \{j|\tilde{A}^{(t+1)}_{i,j}-\tilde{A}^{(t+1)}_{i^*,j}\geq 0\},\quad J^- :=\{j|\tilde{A}^{(t+1)}_{i,j}-\tilde{A}^{(t+1)}_{i^*,j}<0\}.
\end{align*}
Since $\tilde{A}^{(t+1)}$ is row-stochastic, the positive index subset $J^+$ is non-empty. Then, for any $i\in [n]$, we have
\begin{align*}
    0&\leq \boldsymbol{z}^{(t+1)}_i =\tilde{A}^{(t+1)}_{i,\cdot} \boldsymbol{z}^{(t)}-\tilde{A}^{(t+1)}_{i^*,\cdot} \boldsymbol{z}^{(t)}
    =\sum_{1\leq j\leq n} (\tilde{A}^{(t+1)}_{i,j}-\tilde{A}^{(t+1)}_{i^*,j}) \boldsymbol{z}^{(t)}_j \\
    &\leq \sum_{J^+} (\tilde{A}^{(t+1)}_{i,j}-\tilde{A}^{(t+1)}_{i^*,j}) \boldsymbol{z}^{(t)}_j
    \leq \|\boldsymbol{z}^{(t)}\|_{\infty} \sum_{J^+}(\tilde{A}^{(t+1)}_{i,j}-\tilde{A}^{(t+1)}_{i^*,j})    \\
    &= \|\boldsymbol{z}^{(t)}\|_{\infty} \big(\sum_{J^+}\tilde{A}^{(t+1)}_{i,j}-\sum_{J^+}\tilde{A}^{(t+1)}_{i^*,j}\big)\leq (1-\eta)\|\boldsymbol{z}^{(t)}\|_{\infty},
\end{align*}
where the last inequality is because \(\eta\leq \tilde{A}_{i,j}^{(t+1)}\leq 1,~\forall i,j\in[n]\) and \(\sum_{j=1}^n \tilde{A}_{i,j}^{(t+1)}=1\). Then we have
\begin{align*}
    \|\boldsymbol{z}^{(t+1)}\|_\infty \leq (1-\eta)\|\boldsymbol{z}^{(t)}\|_\infty,\quad \forall t=0,1,\dots, 
\end{align*}
and further
\begin{align*}
    \|\boldsymbol{z}^{(T)}\|_\infty \leq (1-\eta)^{T}\|\boldsymbol{z}^{(0)}\|_\infty.
\end{align*}
Therefore, $\boldsymbol{z}^{(t)}\rightarrow 0$ with geometric rate. On the other hand,
\begin{equation}
\begin{aligned}
    0 &\leq p^{(t+1)}-p^{(t)}=  \tilde{A}^{(t+1)}_{i^*,\cdot} \boldsymbol{z}^{(t)}=\sum_{j=1}^n \tilde{A}^{(t+1)}_{i^*,j} \boldsymbol{z}^{(t)}_j \leq \|\boldsymbol{z}^{(t)}\|_\infty \sum_{j=1}^n \tilde{A}^{(t+1)}_{i^*,j}
    \\&\leq \|\boldsymbol{z}^{(t)}\|_\infty \leq (1-\eta)^t\|\boldsymbol{z}^{(0)}\|_\infty,
\end{aligned}    
\label{pt+1-pt}
\end{equation}
implying
\begin{align*}
    p^{(t)}=&p^{(t)}-p^{(t-1)}+p^{(t-1)}-p^{(t-2)}+\dots+p^{(1)}-p^{(0)}+p^{(0)}\notag\\
    \leq & \sum_{i=0}^{t-1} (1-\eta)^i \|\boldsymbol{z}^{(0)}\|_\infty +p^{(0)}=\frac{1-(1-\eta)^t}{\eta}\|\boldsymbol{z}^{(0)}\|_\infty +p^{(0)}\leq \frac{1}{\eta}\|\boldsymbol{z}^{(0)}\|_\infty +p^{(0)}. \label{upper-bound-p}
\end{align*}
Hence, the sequence $p^{(t)}$ is monotonically increasing and bounded, which implies an existent limit $\tilde{p}$.

For any arbitrary $\boldsymbol{x}^{(0)}\in\RR^{n}$, we have $\boldsymbol{z}^{(t)}\rightarrow 0$ and $p^{(t)}\rightarrow \tilde{p}$ (the limit $\tilde{p}$ is related to $\boldsymbol{x}^{(0)}$ and the property of each row-stochastic matrix $\tilde{A}^{(t)}$), which ensures the existence of limit $\prod_{t=1}^{+\infty} \tilde{A}^{(t)}$. Specifically, we choose the arbitrary vector $\boldsymbol{x}^{(0)}$ as $n$ unit vectors $\boldsymbol{e}_1,\dots,\boldsymbol{e}_n$, we can derive,
\begin{align*}
    \prod_{t=1}^{+\infty} \tilde{A}^{(t)}=\lim_{t\rightarrow +\infty} \tilde{A}^{(t)}\cdots \tilde{A}^{(1)} (\boldsymbol{e}_1,\dots,\boldsymbol{e}_n)=(\tilde{p}_1 \one ,\dots,\tilde{p}_n \one ),
\end{align*}
where the exact values of $\tilde{p}_1,\dots,\tilde{p}_n$ cannot be further determined.

Since each matrix $\tilde{A}^{(t)}$ is row-stochastic, the finite product $\prod_{t=1}^T \tilde{A}^{(t)}$ is also row-stochastic, and the limit $\tilde{A}^{\infty}=\prod_{t=1}^{+\infty} \tilde{A}^{(t)}$ is still row-stochastic with identical rows. We denote $\pi:=(\tilde{p}_1,\dots,\tilde{p}_n)^\top$, and then
\begin{equation*}
    \tilde{A}^{\infty}=\one \pi^\top.
\end{equation*}

Since $\tilde{p}_1,\dots,\tilde{p}_n$ are determined by the whole matrix sequence and the exact values are implicit, the stacked limit vector $\pi$ also maintain those properties.

\textbf{Proof of (ii).} Still considering an arbitrary $\boldsymbol{x}^{(0)}\in\RR^{n}$ with decomposition \eqref{two-part-x}, we have
\begin{align}
    (\tilde{A}^{(t)}\cdots \tilde{A}^{(1)})\boldsymbol{x}^{(0)}-\tilde{A}^{\infty}\boldsymbol{x}^{(0)}=\boldsymbol{x}^{(t)}-\tilde{p}\one =\boldsymbol{z}^{(t)}+(p^{(t)}-\tilde{p})\one.\label{AAAx-Ainftyx}
\end{align}
The coefficient of the second term can be bounded as
\begin{align*}
    0\leq \tilde{p}-p^{(t)}=\lim_{T\rightarrow +\infty} p^{(T)}-p^{(t)}\leq \lim_{T\rightarrow +\infty} \sum_{s=t}^{T-1}(1-\eta)^s\|\boldsymbol{z}^{(0)}\|_\infty=\frac{(1-\eta)^t}{\eta}\|\boldsymbol{z}^{(0)}\|_\infty,
\end{align*}
where the second inequality is from \eqref{pt+1-pt}. So the infinite norm of \eqref{AAAx-Ainftyx} can be bounded as 
\begin{align*}
    &\|(\tilde{A}^{(t)}\cdots \tilde{A}^{(1)})\boldsymbol{x}^{(0)}-\tilde{A}^{\infty} \boldsymbol{x}^{(0)} \|_\infty \leq \|\boldsymbol{z}^{(t)}\|_\infty + (\tilde{p}-p^{(t)})\\
    \leq &(1-\eta)^t \|\boldsymbol{z}^{(0)}\|_\infty +\frac{(1-\eta)^t}{\eta}\|\boldsymbol{z}^{(0)}\|_\infty =\frac{1+\eta}{\eta}(1-\eta)^t \|\boldsymbol{z}^{(0)}\|_\infty.
\end{align*}
Substitute \(\boldsymbol{x}^{(0)}\) with unit vectors $\boldsymbol{e}_j,~j=1,2,\dots,n$, and notice that $\|\boldsymbol{z}^{(0)}\|_\infty =\max_{i\in[n]} \boldsymbol{x}^{(0)}_i-\min_{i\in[n]} \boldsymbol{x}^{(0)}_i=1$ for unit vectors $\boldsymbol{e}_j$, we have
\begin{align*}
    \|(\tilde{A}^{(t)}\cdots \tilde{A}^{(1)})_{\cdot,j}-\tilde{p}_j\one \|_\infty \leq \frac{1+\eta}{\eta}(1-\eta)^t.
\end{align*}
Taking all the columns of $\prod_{t=1}^T\tilde{A}^{(t)}$, we obtain the conclusion.
\end{proof}

\subsection{Proof of Proposition~\ref{prop:row-sto-linear-convergence}}
\label{sec-app-limit-property}

The proof pipeline is using the linear convergence result of Lemma~\ref{lem(app):A-limit} in Appendix~\ref{sec(app):A-limit}, which focuses on positive mixing matrices, and then bounding the remaining factors.

\begin{proof}
According to Proposition~\ref{prop:positive}, there exists a positive period integer $B$. Based on the original sequence $\big\{\prod_{i=0}^{k-1}A^{(i)}\big\}_{k=1}^\infty$, by adding brackets to every $B$ factors, we get a new sequence:
\[
\tilde{A}^{(t)}\cdots\tilde{A}^{(1)}:=\bigg(\prod_{i=(t-1)B}^{tB-1} A^{(i)}\bigg)\cdots\bigg(\prod_{i=0}^{B-1}A^{(i)}\bigg),\quad t=1,2,\dots,
\]
where all the $\tilde{A}^{(t)},\,t=1,2,\dots$ satisfy $\tilde{A}_{i,j}^{(t)}\ge \eta >0,\,\forall i,j\in[n]$. According to Lemma~\ref{lem(app):A-limit}, there exists a limit $\tilde{A}^\infty :=\lim_{t\rightarrow\infty}\tilde{A}^{(t)}\cdots\tilde{A}^{(1)}=\one\pi^\top$, and the limit vector $\pi$ is determined by the whole sequence $\{\tilde{A}^{(t)}\}_{t=1}^\infty$. We next prove the original sequence $\big\{\prod_{i=0}^{k-1}A^{(i)}\big\}_{k=1}^\infty$converges to the same limit $\tilde{A}^\infty$. For a given $k>0$, we make the following decomposition,
\begin{equation}
\label{eq:A-decomposition}
\begin{aligned}
\prod_{i=0}^{k-1}A^{(i)}&=\bigg(\prod_{i=TB}^{k-1} A^{(i)}\bigg)\bigg(\prod_{i=(T-1)B}^{TB-1}A^{(i)}\bigg)\cdots\bigg(\prod_{i=0}^{B-1} A^{(i)}\bigg)\\
&:=\tilde{A}^{(-1)}\tilde{A}^{(T)}\cdots \tilde{A}^{(1)},
\end{aligned}
\end{equation}
where $T=\lfloor k/B\rfloor$. Using the linear convergence result from Lemma~\ref{lem(app):A-limit}, we have:
\[
\max\bigg|\prod_{t=1}^T \tilde{A}^{(t)}-\tilde{A}^{\infty}\bigg|\le \frac{1+\eta}{\eta}(1-\eta)^T.
\]
Therefore, the total deviation can be bounded as
\[
\begin{aligned}
    \Bigg\|\prod_{i=0}^{k-1} A^{(i)}-\tilde{A}^{\infty}\Bigg\|_F &= \Bigg\|\tilde{A}^{(-1)}\prod_{t=1}^T\tilde{A}^{(t)}-\tilde{A}^{\infty}\Bigg\|_F = \Bigg\|\tilde{A}^{(-1)}\bigg(\prod_{t=1}^T\tilde{A}^{(t)}-\tilde{A}^\infty \bigg)\Bigg\|_F\\
    &\le \Big\|\tilde{A}^{(-1)}\Big\|_2\Bigg\|\prod_{t=1}^T\tilde{A}^{(t)}-\tilde{A}^\infty\Bigg\|_F \le \sqrt{n} \cdot n\max\bigg|\prod_{t=1}^T \tilde{A}^{(t)}-\tilde{A}^\infty\bigg|\\
    & \le n^{3/2}\frac{1+\eta}{\eta}(1-\eta)^T,
\end{aligned}
\]
Without loss of generality, we let $k>B$, so that $T=\lfloor k/B\rfloor \ge k/B-1$. The norm can be bounded as 
\[
\Bigg\|\prod_{i=0}^{k-1} A^{(i)}-\tilde{A}^{\infty}\Bigg\|_F \le n^{3/2}\frac{1+\eta}{\eta(1-\eta)}\cdot\big(\sqrt[B]{1-\eta}\big)^k.
\]
Therefore, the limit $\prod_{i=0}^{\infty}A^{(i)}=\one\pi^\top$ exists and the value is given by the infinite product of $\{\tilde{A}^{(t)}\}_{t=1}^\infty$. Further, the convergence is exponentially fast.
\end{proof}

\subsection{Proof of the Constant-\texorpdfstring{$R_k$}{Rk} Counterexample}\label{app:counterexample}

Throughout, $n=2$, $d=1$, $\va:=(1,3)^\top$, $\vb:=(0,1)^\top$ and
$\Lambda:=\mathrm{Diag}(\va)$. By~\eqref{eq:ce-A}, $\nabla f_1(x)=x$ and
$\nabla f_2(x)=3(x-1)$, so the stacked gradient is
$\vg^{(k)}=\Lambda(\vx^{(k)}-\vb)$ and $\nabla f(\bar x)=2\bar x-\frac32$.
We write $\pi:=(\frac23,\frac13)^\top$, which satisfies $\pi^\top A=\pi^\top$ and
$\pi^\top\one=1$.

\begin{proof}
\textbf{Step 1 (reduction to a time-invariant affine map).} Since $A$ is static,
every outer iteration uses the same matrices: the gossip product is $A^R$, and the
PULM matrix $W_R:=W^{(k,R)}$ is independent of $k$. Hence~\eqref{eq:PULM-DGD-iter}
becomes
\begin{equation}\label{eq:ce-affine}
\vx^{(k+1)}=A^R\vx^{(k)}-\gamma W_R\Lambda\big(\vx^{(k)}-\vb\big)
=M\vx^{(k)}+\gamma W_R\Lambda\vb,\qquad M:=A^R-\gamma W_R\Lambda.
\end{equation}

\textbf{Step 2 (the PULM matrix).} For $r\ge1$ the adjust
step~\eqref{eq:pulm-dgd-3} pins both diagonal entries of $W^{(k,r)}$ at
$\frac12$, so $W^{(k,r)}=\begin{pmatrix}1/2&p_r\\q_r&1/2\end{pmatrix}$. Reading
off the off-diagonal entries of $AW^{(k,r)}$ gives
$p_{r+1}=\frac45p_r+\frac1{10}$ and $q_{r+1}=\frac35q_r+\frac15$, and the first
round gives $p_1=\frac15$ and $q_1=\frac25$. Both recursions are affine with fixed
point $\frac12$, so
\begin{equation}\label{eq:ce-W}
W_R=\begin{pmatrix}\frac12&\frac12-\frac{3}{10}u\\[2pt]\frac12-\frac1{10}v&\frac12\end{pmatrix},
\qquad u=\big(\tfrac45\big)^{R-1},\ v=\big(\tfrac35\big)^{R-1}.
\end{equation}
Thus $W_R$ differs from $\frac12\one\one^\top$ by a residual averaging error of
order $(\frac45)^{R}$; this error is the source of the gap.

\textbf{Step 3 (the case $R=1$).} A fixed point of~\eqref{eq:ce-affine} solves
$N\vx^{\infty}=\gamma W_R\Lambda\vb$ with $N:=I-A^R+\gamma W_R\Lambda$. For
$R=1$,~\eqref{eq:ce-W} gives
\[
N=\begin{pmatrix}\frac15+\frac{\gamma}{2}&-\frac15+\frac{3\gamma}{5}\\[2pt]
-\frac25+\frac{2\gamma}{5}&\frac25+\frac{3\gamma}{2}\end{pmatrix},
\qquad
\gamma W_1\Lambda\vb=\begin{pmatrix}\frac{3\gamma}{5}\\[2pt]\frac{3\gamma}{2}\end{pmatrix},
\qquad
\det N=\frac{\gamma}{100}(82+51\gamma)>0,
\]
so the fixed point is unique, and Cramer's rule gives the expression
for $\vx^{\infty}$ in~\eqref{eq:ce-xinf}. Consequently
$\nabla f(\bar x^{\infty})=x_1^{\infty}+x_2^{\infty}-\frac32
=-\frac{51\gamma+30}{2(51\gamma+82)}$. This quantity is increasing in magnitude on
$\gamma>0$, so $|\nabla f(\bar x^{\infty})|\ge\frac{30}{164}=\frac{15}{82}$.

It remains to show $\rho(M)<1$ for $\gamma\in(0,1]$. A direct computation gives
$\operatorname{tr}M=\frac75-2\gamma$ and
$\det M=\frac25-\frac{59}{50}\gamma+\frac{51}{100}\gamma^2$. For a real
$2\times2$ matrix with characteristic polynomial
$\chi(\lambda)=\lambda^2-(\operatorname{tr}M)\lambda+\det M$, the condition
$\rho(M)<1$ is equivalent to the Jury conditions $\chi(1)>0$, $\chi(-1)>0$ and
$|\det M|<1$. First, $\chi(1)=\det N>0$. Second,
$\chi(-1)=\frac{14}{5}-\frac{159}{50}\gamma+\frac{51}{100}\gamma^2$ is a convex
quadratic whose smaller root $\frac{53}{17}-\frac{\sqrt{11001}}{51}$ exceeds $1$, so
$\chi(-1)>0$ on $(0,1]$. Third, $\det M$ is a convex quadratic in $\gamma$ with vertex at
$\gamma=\frac{59}{51}>1$, so it decreases on $(0,1]$ from $\frac25$ to
$-\frac{27}{100}$, and $|\det M|<1$ there.
Hence $\rho(M)<1$, and $\vx^{(k)}\to\vx^{\infty}$ from every initialization.

\textbf{Step 4 (general $R$: convergence for small $\gamma$).} At $\gamma=0$,
$M=A^R$ has the simple eigenvalue $1$, with left and right eigenvectors $\pi$ and
$\one$, and the eigenvalue $(\frac25)^R$. The simple eigenvalue $\lambda(\gamma)$
of $M$ near $1$ is real and differentiable in $\gamma$, with
$\lambda'(0)=-\pi^\top W_R\Lambda\one=-\pi^\top W_R\va<0$, since $W_R$ and $\va$
are entrywise positive. Hence there exists $\gamma_R>0$ such that
$\rho(M)<1$ for all $\gamma\in(0,\gamma_R)$, and $\vx^{(k)}$ converges to the
unique fixed point $\vx^{\infty}(\gamma)=N^{-1}\gamma W_R\Lambda\vb$.

\textbf{Step 5 (general $R$: the limit as $\gamma\to0^+$).} Left-multiplying
$N\vx^{\infty}(\gamma)=\gamma W_R\Lambda\vb$ by $\pi^\top$ annihilates
$I-A^R$; cancelling $\gamma$ gives
\begin{equation}\label{eq:ce-pi}
\pi^\top W_R\Lambda\,\vx^{\infty}(\gamma)=\pi^\top W_R\Lambda\,\vb
\qquad\text{for every }\gamma\in(0,\gamma_R).
\end{equation}
By Cramer's rule, the entries of $\vx^{\infty}(\gamma)$ are ratios of
determinants; each numerator is $O(\gamma)$, because one column of $N$ is replaced
by $\gamma W_R\Lambda\vb$, and the denominator $\det N$ has a
positive coefficient of $\gamma$, namely $(1-(\frac25)^R)\,\pi^\top W_R\va$,
because $I-A^R$ has rank one. Hence $\vx^{\infty}(\gamma)$ has a finite limit
$\vx^{\infty}_0$ as $\gamma\to0^+$. Passing to the limit in
$N\vx^{\infty}(\gamma)=\gamma W_R\Lambda\vb$ gives $(I-A^R)\vx^{\infty}_0=0$, so
$\vx^{\infty}_0=x_0\one$ for a scalar $x_0=x_0(R)$, since the eigenvalue $1$ of
$A^R$ is simple. Letting $\gamma\to0^+$ in~\eqref{eq:ce-pi} and using
$\Lambda\one=\va$ and $\Lambda\vb=(0,3)^\top$,
\[
x_0=\frac{\pi^\top W_R(0,3)^\top}{\pi^\top W_R\va}
=\frac{2p_R+\frac12}{2p_R+\frac13q_R+\frac56},
\]
where $p_R=\frac12-\frac3{10}u$ and $q_R=\frac12-\frac1{10}v$ by~\eqref{eq:ce-W}.
Therefore
\[
\nabla f(x_0)=2x_0-\frac32
=\frac{p_R-\frac12q_R-\frac14}{2p_R+\frac13q_R+\frac56}
=\frac{v/10-3u/5}{4-6u/5-v/15},
\]
which is the identity in~\eqref{eq:ce-gap}.

\textbf{Step 6 (bounds).} Since $0<v\le u\le1$, the numerator satisfies
$-\frac{3u}{5}\le\frac{v}{10}-\frac{3u}{5}\le-\frac{u}{2}<0$, and the denominator
lies in $[\frac{41}{15},4]$. Hence
$\frac{u}{8}\le|\nabla f(x_0)|\le\frac{9u}{41}\le u$, which is the bound
in~\eqref{eq:ce-gap}. Since $\vx^{\infty}(\gamma)$ depends continuously on
$\gamma$, the limit is not stationary for all sufficiently small $\gamma$.
\end{proof}
\bibliographystyle{plainnat}
\bibliography{references}

\end{document}